\documentclass[10pt]{amsart}

\usepackage[utf8]{inputenc}

\usepackage{graphicx}
\usepackage{amsfonts}
\usepackage{hyphenat}
\newtheorem{theorem}{Theorem}[section]
\newtheorem{lemma}[theorem]{Lemma}
\newtheorem{remark}[theorem]{Remark}
\newtheorem{proposition}[theorem]{Proposition}
\newtheorem{corollary}[theorem]{Corollary}
\newtheorem{definition}[theorem]{Definition}

\usepackage[utf8]{inputenc}
\usepackage[table]{xcolor}
\usepackage{lscape}
\usepackage{pgfplots}
\usepackage{subcaption}
\usepackage{tikz}
\usepackage{pgfplots}
\usepackage{adjustbox}
\usepackage{tikz-cd}
\usepackage{tikz-3dplot}
\tdplotsetmaincoords{60}{115}
\pgfplotsset{compat=newest}
\title{Stable submanifolds in the product of  Projective Spaces}
\author{Alejandra Ramirez-Luna}
\begin{document}

\begin{abstract}
We provide a classification theorem for compact stable minimal immersions (CSMI) of codimension $1$ or dimension $1$ (codimension $1$ and $2$ or dimension $1$ and $2$) in the product of a complex (quaternionic) projective space with any other Riemannian manifold. We characterize the complex minimal immersions of codimension $2$ or dimension $2$ as the only CSMI in the product of two complex projective spaces. As an application, we characterize the CSMI of codimension $1$ or dimension $1$ (codimension $1$ and $2$ or dimension $1$ and $2$) in the product of a complex (quaternionic) projective space with any compact rank one symmetric space.
\end{abstract}
\maketitle

\begin{section}{Introduction }
Let $M$ be a Riemannian manifold of dimension $n+d$. It is a really interesting problem to know what are the submanifolds $\Sigma$ of dimension $n$ of $M$ that  minimizes area under perturbations.  For example in the Euclidean space $\mathbb{R}^{3}$, in Fig.\ref{perturbationplane} we can see 
intuitively that the blue plane $z=0$ has less area than the perturbation, in the figure represented by the color degradation (see \cite{do1979stable} and \cite{chodosh2021stable}).
\begin{figure}
\begin{centering}
\begin{tikzpicture}
\begin{axis}
\addplot3 [surf] {-0.1};
\addplot3[
    surf,
]
{exp(-0.2*x^2-0.2*y^2)};
\end{axis}
\end{tikzpicture}
\end{centering}
\caption{Perturbation of plane in $\mathbb{R}^{3}$}
\label{perturbationplane}
\end{figure}
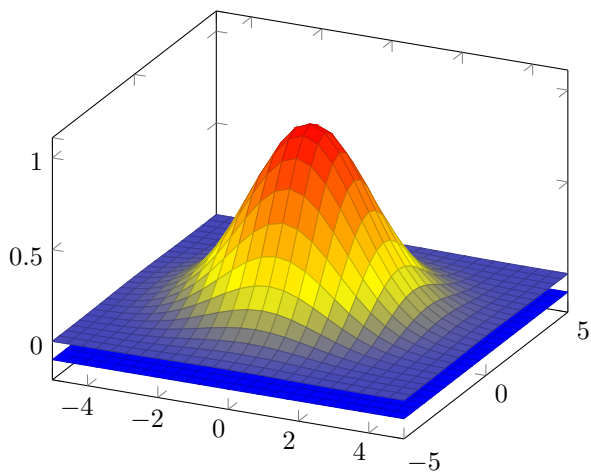

Since it is a minimization problem, one condition $\Sigma$ must satisfy is that it be a critical point of the area functional (i.e. minimal).  Returning to the case of the plane in $\mathbb{R}^{3}$, we can see then that the plane is minimal. But contrary to the previous example, there are many examples where the submanifold is minimal and it does not minimize area. For example, if we perturb an equator, $E$, in the Euclidean sphere along normal direction with constant height, we get a circle $P(E)$ with smaller radius, and thus, smaller length (see Fig.\ref{persp}). 

\begin{figure}
\begin{tikzpicture}[tdplot_main_coords, scale = 2.5]
 
\shade[ball color=green!10!white,opacity=0.4,
    opacity = 0.4
] (0,0,0) circle (1cm);

\tdplotsetrotatedcoords{0}{0}{0};
\draw[,
	tdplot_rotated_coords,
	black
] (0,0,0) circle (1);

\tdplotsetrotatedcoords{0}{0}{0};
\draw[,
	tdplot_rotated_coords,
	red
] (0,0,0.2) circle (0.98) ;

\node [above, text=red](A) at (-0.48,-0.6,0.35) {$P(E)$};

\node [below](B) at (1,0,0) {$E$};
 
\draw[-stealth] (0,0,0) -- (1.80,0,0) 
    node[below left] {$x$};
 
\draw[-stealth] (0,0,0) -- (0,1.30,0)
    node[below right] {$y$};
 
\draw[-stealth] (0,0,0) -- (0,0,1.30)
    node[above] {$z$};
 
\draw[dashed, gray] (0,0,0) -- (-1,0,0);
\draw[dashed, gray] (0,0,0) -- (0,-1,0);
\end{tikzpicture}
\caption{Perturbation $P(E)$ of equator $E$ in $S^{2}$.}
\label{persp}
\end{figure}
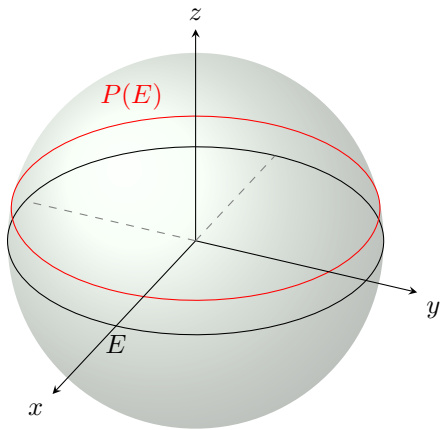

 Therefore, we have to study the second variation of the area functional. More specifically, we need the second variation to be non negative for all possible perturbations of $\Sigma$. If $\Sigma$ is as described above, we say that $\Sigma$ is stable. The second variation of the area functional defines an operator acting on normal sections of $\Sigma$ (see Preliminaries), called the Jacobi or stability operator. The number of negative eigenvalues counting multiplicity is what we call the Morse index. Then, $\Sigma$ is stable if its Morse index is zero.
 
 For a long time, it has been object  of study if a Riemannian manifold has stable submanifolds and if the case to characterize them. In particular when the Riemannian manifold is well known, as the Euclidean space, the Euclidean sphere, and projective spaces. For example, in the $3$-dimensional Euclidean space , Do Carmo and Peng \cite{do1979stable} proved that planes are the only stable complete minimal surfaces in $\mathbb{R}^{3}$. Recently, in the same direction, Chodosh and Li proved that a complete, two-sided, stable minimal hypersurface in $\mathbb{R}^{4}$ must be flat \cite{chodosh2021stable}. For the case of compact stable minimal immersions (CSMI), Simons in \cite{simons1968minimal} proved that there are no CSMI in the Euclidean sphere. 
\begin{theorem}\cite{simons1968minimal}
\label{t1.1}
Let $\Sigma$ be a compact, minimal $n$-dimensional submanifold immersed in $S^{n+d}$. Then, the index of $\Sigma$ is greater than or equal to $d$, and equality holds only when $\Sigma$ is $S^{n}$. 
\end{theorem}

He used the fact that $S^{n+d}$ is an hypersurface of $\mathbb{R}^{n+d+1}$. More precisely, the stability operator was evaluated on the projections $w$ of the constant parallel vector fields $v\in \mathbb{R}^{n+d+1}$ in $N_{\Sigma}^{S^{n+d}}$, where in  general $N_{\Sigma}^{M}$ denote the normal space of $\Sigma$ in $M$ (see Fig. \ref{ESF}). 

\begin{figure}
\begin{tikzpicture}[tdplot_main_coords, scale = 2.5]
\coordinate (Q) at ({0},{0},{1});

\shade[ball color = lightgray,
	opacity = 0.5
] (0,0,0) circle (1cm);

\draw[-stealth] (0,0,0) -- (1.80,0,0) 
	node[below left] {$x$};

\draw[-stealth] (0,0,0) -- (0,1.30,0)
	node[below right] {$y$};

\draw[-stealth] (0,0,0) -- (0,0,1.30)
	node[above] {$z$};

\draw[dashed, gray] (0,0,0) -- (-1,0,0);
\draw[dashed, gray] (0,0,0) -- (0,-1,0);

\draw[fill = red!50] (Q) circle (0.5pt);

\tdplotsetrotatedcoords{0}{0}{0};
\draw[dashed,
	tdplot_rotated_coords,
	gray
] (0,0,0) circle (1);

\tdplotsetrotatedcoords{90}{90}{90};
\draw[
	tdplot_rotated_coords,
	blue
] (1,0,0) arc (0:360:1);

\fill[pink!50,opacity=0.4] (-1.2,1.2,1) -- (1.2,1.2,1) -- (1.2,-1.2,1) -- (-1.2,-1.2,1) -- (-1.2,1.2,1);

\draw[gray, thick] (0,0,1) -- (0,1.2,1);
\draw[gray, thick] (0,0,1) -- (0,-1.2,1) node[above] {$N_{\Sigma}^{S^{2}}$};


\color{blue}
\node [below](hip) at (0,0,-1.01) {$\Sigma$};
\color{black}

\draw[red, thick, -stealth] (0,0,1) -- (1.2,0,1) node[left] {$T\Sigma$};

\coordinate (P) at ({-1/sqrt(3)},{0.791623},{0.2});

\draw[green, thick, -stealth] (0,0,0) -- (P) node[right] {$v$};

\coordinate (V) at ({-1/sqrt(3)},{0.791623},{1+0.2});

\draw[green, -stealth] (0,0,1) -- (V) node[right] {$v$};
\draw[thick, -stealth] (0,0,1) -- (0,0.791623,1) node[below] {$w$};
\draw[thin, dashed] (V) --(0,0.791623,1);
\end{tikzpicture}
\caption{Normal sections.}
\label{ESF}
\end{figure}
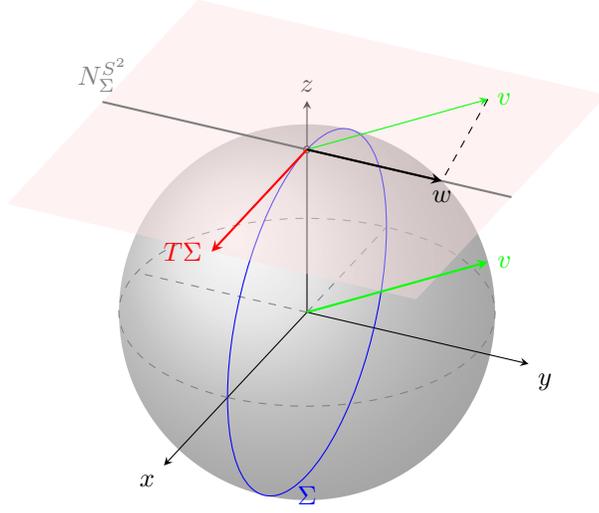

Afterwards, Lawson and Simons in \cite{lawson1973stable} characterized the complex submanifolds (in the sense that each tangent space of the submanifold is invariant under the complex structure) as the only CSMI in the complex projective spaces. This was followed by Ohnita \cite{ohnita1986stable}, who completed the classification of CSMI in all compact rank one symmetric spaces,

 \begin{theorem}\cite{ohnita1986stable}
 \label{ohnita0}
 Let $\Sigma$ be a compact minimal $n$-dimensional submanifold immersed in the real projective space $\mathbb{R}P^{n+d}$ with the standard metric. Then, $\Sigma$ is stable if and only if $\Sigma$ is a real projective subspace $\mathbb{R}P^{n}$ of $\mathbb{R}P^{n+d}$.
 \end{theorem}

\begin{theorem}\cite{lawson1973stable,ohnita1986stable}
\label{ohnita}
Let $\Sigma$ be a compact stable minimal $n$-dimensional submanifold immersed in $\mathbb{F}P^{m}$, where $\mathbb{F}=\mathbb{C},\mathbb{H}, \mathbb{O}$ (complex, quaternion, octonion numbers, respectively). Then,
\begin{itemize}
    \item If $\mathbb{F}=\mathbb{C}$, then $n=2l$ for some integer $l$ and $\Sigma$ is a complex submanifold, in the sense that each tangent space is invariant under the complex structure of the complex projective space (Lawson and Simons).
    \item If $\mathbb{F}=\mathbb{H}$, then $n=4l$ for some integer $l$ and $\Sigma$ is a quaternionic projective subspace $\mathbb{H}P^{l}$ of $\mathbb{H}P^{m}$ (Ohnita).
    \item If $\mathbb{F}=\mathbb{O}$ and $m=2$, then $n=8$ and $\Sigma$ is a Cayley projective line (Ohnita).
\end{itemize}

\end{theorem}

The projective space $\mathbb{F}P^{m}$ can be isometrically immersed in some Euclidean space $\mathbb{R}^{l}$ by the generalized Veronese imbedding (Sakamoto \cite{sakamoto1977planar}). Therefore, in order to prove the previous theorem, the stability operator was evaluated on the projections of the constant parallel vector fields  $v\in \mathbb{R}^{l}$ in $N_{\Sigma}^{\mathbb{F}P^{m}}$, as Simons did in the proof of Theorem \ref{t1.1}. This was done because after running $v$ in an orthonormal basis of  $\mathbb{R}^{l}$, a term $T$ given by the sum of the second variations of the associated normal sections is non-positive. Since $T$ is also non-negative  (because the submanifold is stable),  $T$ must be zero. Thus, using that $T=0$, we can obtain geometric information about the submanifold. 

After classifying the CSMI of these well known Riemannian manifolds, a natural direction is to study CSMI in their products. Along these lines, Torralbo and Urbano proved a classification theorem of CSMI in the product of a sphere and any Riemannian manifold whenever the dimension of the sphere is at least three or the immersion has codimension $1$. 

 \begin{theorem} \cite{torralbo2014stable}
 \label{TU}
 Let $M$ be any Riemannian manifold and $\Phi=(\phi,\psi):\Sigma \rightarrow \mathbb{S}^{m}\times M$ be a compact minimal immersion of dimension $n$ in $ \mathbb{S}^{m}\times M$, $n\geq 2$, satisfying either $m\geq3$ or $m=2$ and $\Phi$ is of codimension $1$. Then, $\Phi$ is stable if and only if

 \begin{itemize}
     \item $\Sigma=S^{m}$ and $\Phi(\Sigma)$ is a slice $S^{m}\times \{q\}$ with $q$ a point in $M$.
     \item $\Sigma$ is a covering of $M$ and $\Phi(\Sigma)$ is a slice $\{p\}\times M$ with $p$ a point of $S^{m}$.
     \item $\psi:\Sigma \rightarrow M$ is a stable minimal submanifold and $\Phi(\Sigma)$ is $\{p\}\times \psi(\Sigma)$ with $p$ a point of $S^{m}$.
     
     \item $\Sigma=S^{m}\times \bar{\Sigma}$, $\Phi=(Id,\psi)$ and $\psi:\bar{\Sigma}\rightarrow M$ is a stable minimal submanifold. 
 \end{itemize}
 \end{theorem}

 Moreover, they complete the classification of CSMI in the product of two spheres. 
 
 \begin{theorem}  \cite{torralbo2014stable}
 \label{TU2}
 Let $\Phi=(\phi,\psi):\Sigma \rightarrow S^{n_{1}}(r_{1})\times S^{n_{2}}(r_{2})$ be a compact minimal immersion of dimension $n$ in $S^{n_{1}}(r_{1})\times S^{n_{2}}(r_{2})$, $n\geq 2$. Then, $\Phi$ is stable if and only if one of the following possibilities occurs
 
 \begin{itemize}
     \item $\Sigma=S^{n_{1}}(r_{1})$ and $\Phi(\Sigma)$ is a slice $S^{n_{1}}(r_{1})\times \{q\}$ with $q$ a point of $S^{n_{2}}(r_{2})$.
     \item $\Sigma=S^{n_{2}}(r_{2})$ and $\Phi(\Sigma)$ is a slice $ \{p\} \times S^{n_{2}}(r_{2})$ with $p$ a point of $S^{n_{1}}(r_{1})$.
     \item $n_{1}=n_{2}=n=2$, $\Sigma$ is orientable and $\Phi$ is a complex immersion of the Riemann surface $\Sigma$ in $S^{2}(r_{1})\times S^{2}(r_{2})$ with respect to one of the two complex structures that $S^{2}(r_{1})\times S^{2}(r_{2})$ has. 
 \end{itemize}
 \end{theorem}

To prove the last two theorems, the stability operator was evaluated on the projections of the vector fields $(v,0)$ to the normal part of the immersion in the Riemannian product, where $v$ are the constant parallel vector fields used in the proof of Theorem \ref{t1.1}.

In \cite{chen2013stable}, Chen and Wang generalize Theorem \ref{TU} to the product of any hypersurface  $M$ of the Euclidean space $\mathbb{R}^{m_{1}+1}$ with certain conditions, and any Riemannian manifold. Similar to Torralbo and Urbano, the stability operator was evaluated on projections of the vector fields $(v,0)$ to the normal part of the immersion in the Riemannian product, where $v\in \mathbb{R}^{m_{1}+1}$ are the constant parallel vector fields in the Euclidean space.

 \begin{theorem}\cite{chen2013stable}
 \label{chen}
 Let $\Phi=(\phi,\psi):\Sigma \rightarrow \bar{M}:=M_{1}\times M_{2}$ be a compact minimal immersion in $\bar{M}$, where $M_{1}$ is a compact connected hypersurface in $\mathbb{R}^{m_{1}+1}$ and $M_{2}$ is any Riemannian manifold. Assume that the sectional curvature $K_{M_{1}}$ of $
M_{1}$ satisfies
\begin{center}
    $\frac{1}{\sqrt{m_{1}-1}}\leq K_{M_{1}} \leq 1$.
\end{center}
 Then, $\Phi$ is stable if and only if
 \begin{itemize}
     \item $\Sigma=M_{1}$ and $\Phi(\Sigma)$ is a slice $M_{1}\times \{q\}$ with $q$ a point in $M_{2}$.
     \item $\Sigma$ is a covering of $M_{2}$ and $\Phi(\Sigma)$ is a slice $\{p\}\times M_{2}$ with $p$ a point of $M_{1}$.
     \item $\psi:\Sigma \rightarrow M_{2}$ is a stable minimal submanifold and $\Phi(\Sigma)$ is $\{p\}\times \psi(\Sigma)$ with $p$ a point of $M_{1}$.
     \item $\Sigma=M_{1}\times \bar{\Sigma}$, $\Phi=(Id,\psi)$ and $\psi:\bar{\Sigma}\rightarrow M_{2}$ is a stable minimal submanifold. 
 \end{itemize}
 \end{theorem}
 
This motivates us to evaluate the stability operator on the projections of the vector fields $(v,0)$ to the normal part of an immersion in the Riemannian product of a complex or quaternionic projective  space with any other Riemannian manifold, where $v$ are the constant parallel vector fields used in the Theorem \ref{ohnita}. Notice that we can not use Theorem \ref{chen} because the complex and quaternionic projective spaces are not hypersurfaces of Euclidean space (see \cite{MR232396, sanderson1963non}). For $\mathbb{C}P^{\frac{m_{1}}{2}}\times M$, we have the following:

\begin{theorem}
\label{COMPLEXTHEOREM}
Let $\Phi=(\psi,\phi):\Sigma \rightarrow \mathbb{C}P^{\frac{m_{1}}{2}}\times M$ be a compact stable minimal immersion of codimension $d$ and dimension $n$, where $M$ is any Riemannian manifold of dimension $m_{2}$. Then, 
\begin{itemize}
    \item If $d=1$, $\Sigma=\mathbb{C}P^{\frac{m_{1}}{2}}\times \hat{\Sigma}$, $\Phi=Id \times \hat{\phi}$ where $\hat{\phi}:\hat{\Sigma}\rightarrow M$ is a stable minimal immersion of codimension $1$, and therefore $\Phi(\Sigma)=\mathbb{C}P^{\frac{m_{1}}{2}}\times \hat{\phi}(\hat{\Sigma})$. In particular, for $m_{2}=1$, $\Sigma=\mathbb{C}P^{\frac{m_{1}}{2}}$, $\hat{\phi}$ is a constant function, and $\Phi(\Sigma)=\mathbb{C}P^{\frac{m_{1}}{2}}\times \{q\}$, for $q\in M$. 
    \item If n=1, $\phi:\Sigma \rightarrow M$ is a stable geodesic, $\psi$ is a constant function, and therefore $\Phi(\Sigma)=\{r\}\times \phi(\Sigma)$ for $r\in \mathbb{C}P^{\frac{m_{1}}{2}}$.
\end{itemize}
\end{theorem}

\begin{remark} We have the following:
\begin{itemize}
    \item For the particular case $m=2$ in Theorem \ref{TU}, Torralbo and Urbano classified the CSMI of codimension $d=1$ in $S^{2}\times M$. Applying Theorem \ref{COMPLEXTHEOREM} for $m_{1}=2$ and using the fact that $S^{2}$ is isometric to $\mathbb{C}P^{1}$, we obtain the same classification result.
    \item Theorem \ref{COMPLEXTHEOREM} gives a complete the classification of CSMI in $\mathbb{C}P^{1}\times M$, where $M$ is a Riemannian manifold of dimension $1$. 
\end{itemize}
\end{remark}

The last theorem tells us that the CSMI of either codimension $1$ or dimension $1$ in the Riemannian product  $\mathbb{C}P^{\frac{m_{1}}{2}}\times M$ are products of trivial CSMI of  $\mathbb{C}P^{\frac{m_{1}}{2}}$  with CSMI of  $M$. As an aplication, 
since we know the stable minimal submanifolds of $M=S^{s}, \mathbb{R}P^{s}, \mathbb{C}P^{s}, \mathbb{H}P^{s}, \mathbb{O}P^{2}$, we have the following corollaries:

\begin{corollary} 
\label{A2}
There are no compact stable minimal immersions of
\begin{itemize}
    \item codimension $d=1$ in the product manifold  $\mathbb{C}P^{\frac{m_{1}}{2}}\times S^{s}$,  $\mathbb{C}P^{\frac{m_{1}}{2}}\times \mathbb{O} P^{2}$, or  $\mathbb{C}P^{\frac{m_{1}}{2}}\times \mathbb{K}P^{s}$ other than  $\mathbb{C}P^{\frac{m_{1}}{2}}\times \{q\}$ in $\mathbb{C}P^{\frac{m_{1}}{2}}\times S^{1}$, for $q\in S^{1}$,
    \item or dimension $n=1$ in the product manifold  $\mathbb{C}P^{\frac{m_{1}}{2}}\times S^{s}$,  $\mathbb{C}P^{\frac{m_{1}}{2}}\times \mathbb{O} P^{2}$, or  $\mathbb{C}P^{\frac{m_{1}}{2}}\times \mathbb{K}P^{s}$ other than  $\{r\}\times S^{1}$
in $\mathbb{C}P^{\frac{m_{1}}{2}}\times S^{1}$, for $r\in \mathbb{C}P^{\frac{m_{1}}{2}}$,
\end{itemize}
where $\mathbb{K}\in \{\mathbb{C}, \mathbb{H}\}$. 
\end{corollary}

\begin{remark}
The particular case in the previous corollary of CSMI of codimension $d=1$ in  $\mathbb{C}P^{\frac{m_{1}}{2}}\times S^{s}$ for $s\geq 2$, can be also obtained as a consequence of Theorem \ref{TU} from  Torralbo and Urbano by setting $M=\mathbb{C}P^{\frac{m_{1}}{2}}$ and applying Theorem \ref{ohnita}.
\end{remark}

\begin{corollary}
\label{A3}
The only compact stable minimal immersion of 
\begin{itemize}
    \item codimension $d=1$ in the product space $\mathbb{C}P^{\frac{m_{1}}{2}}\times \mathbb{R}P^{s}$ is $\mathbb{C}P^{\frac{m_{1}}{2}}\times \mathbb{R}P^{s-1}$,
    \item or  dimension $n=1$ in the product space $\mathbb{C}P^{\frac{m_{1}}{2}}\times \mathbb{R}P^{s}$ is
$\{r\}\times \mathbb{R}P^{1}$, $r\in \mathbb{C}P^{\frac{m_{1}}{2}}$.
\end{itemize}
\end{corollary}

The first item in Theorem \ref{ohnita} tells us that the CSMI in the complex projective space behave well under the complex structure of the complex projective space. Therefore, it is expected (see \cite{lawson1973stable}, and \cite{torralbo2014stable} and references therein) that if the manifold has a complex structure, named $J$, then the CSMI in this manifold also behave well under $J$. In the specific case of $\mathbb{C}P^{\frac{m_{1}}{2}}\times M$, where $M$ is an arbitrary manifold, we do not know if this product manifold has a complex structure. Thus, there is no natural complex structure for the minimal submanifold to be well behaved with.  
However, we can indeed expect the CSMI to have a complex behaviour in the projections in the first component (the component associated to the complex projective space) of some vectors associated to the immersion. In fact, we have the following Lemma, 

\begin{lemma}
\label{complexdc2}
Under the same conditions than Theorem \ref{COMPLEXTHEOREM} we have:
\begin{itemize}
     \item If $d=2$, $\eta_{2}^{1}=\pm J(\eta_{1}^{1})$, where $\{\eta_{1},\eta_{2}\}$ is an orthonormal basis of $N_{p}\Sigma$, $ p \in \Sigma$.
    \item If $n=2$, $e_{2}^{1}=\pm J(e_{1}^{1})$, where $\{e_{1},e_{2}\}$ is an orthonormal basis of $T_{p}\Sigma$, $ p \in \Sigma$.
\end{itemize}
Here, $J$ is the complex structure of $\mathbb{C}P^{\frac{m_{1}}{2}}$ and $w^{1}$ is the projection of $w$ in $T_{\psi(p)}\mathbb{C}P^{\frac{m_{1}}{2}}$.
\end{lemma}

If $M=\mathbb{C}P^{\frac{m_{2}}{2}}$, the Riemannian product $\mathbb{C}P^{\frac{m_{1}}{2}}\times \mathbb{C}P^{\frac{m_{2}}{2}}$ has two complex structures, $J_{1}$ and $J_{2}$ induced by the complex structure $J$, of the complex projective space (see Definition \ref{structures}). Therefore, according to what was mentioned before, it is expected that the CSMI in $\mathbb{C}P^{\frac{m_{1}}{2}}\times \mathbb{C}P^{\frac{m_{2}}{2}}$ behave well under a complex structure of $\mathbb{C}P^{\frac{m_{1}}{2}}\times \mathbb{C}P^{\frac{m_{2}}{2}}$. In fact, Lemma  \ref{complexdc2} gives us information about how the complex structure of the complex projective space behaves on projections onto the first component of normal or tangent vectors. Using the same technique as in the proof  of Lemma \ref{complexdc2} when $M=\mathbb{C}P^{\frac{m_{2}}{2}}$, we can get information about how the complex structure of the complex projective space behaves on projections onto the second component of normal or tangent vectors. Using the behaviours of those projections under $J$,  we can determine that for all points $p$ in $\Sigma$, $T_{p}\Sigma$ has the structure $J_{1}$ or the structure $J_{2}$. We then prove that in fact every point has the same complex structure. More precisely, we proceed by contradiction and we assume there are two points with different structures. This allows us to construct a real function $g:(t-\epsilon,t+\epsilon)\rightarrow \mathbb{R}$ such that it changes sign and vanishes to infinite order at $t$. We prove that, in our setting, the function $g$ is real analytic, therefore $g$ must vanish in the interval, which contradicts the change of sign. 

\begin{theorem}
\label{M3}
The only compact stable minimal immersions of codimension $d=2$ or dimension $n=2$ in the product manifold $\bar{M}:=\mathbb{C}P^{\frac{m_{1}}{2}}\times \mathbb{C}P^{\frac{m_{2}}{2}}$ are the complex ones, in the sense that each tangent space is invariant under the complex structure $J_{1}$ or $J_{2}$ of $\bar{M}$ (see Definition \ref{structures}). 
\end{theorem}

\begin{remark} We have the following:
\begin{itemize}
\item For the particular case $n_{1}=n_{2}=2$ in Theorem \ref{TU2}, Torralbo and Urbano characterized the CSMI of dimension $2 \leq n \leq 3$ in $S^{2}\times S^{2}$. Applying Corollary \ref{A2}, Theorem \ref{M3} and using again the fact that $\mathbb{C}P^{1}$ is isometric to $S^{2}$, we obtain the same characterization. 
\item Corollary \ref{A2} and Theorem \ref{M3} give a complete characterization of CSMI in $\mathbb{C}P^{1}\times \mathbb{C}P^{1}$.
\end{itemize}
\end{remark}

On the other hand, for $\mathbb{H}P^{\frac{m_{1}}{4}}\times M$, we have the following:

\begin{theorem}
\label{QUATERNIONICTHEOREM}
Let $\Phi=(\psi,\phi):\Sigma \rightarrow \mathbb{H}P^{\frac{m_{1}}{4}}\times M$ be a compact stable minimal immersion of codimension $d$ and dimension $n$, where $M$ is any Riemannian manifold of dimension $m_{2}$. Then, 

\begin{itemize}
    \item If $d=1$,
     $\Sigma=\mathbb{H}P^{\frac{m_{1}}{4}}\times \hat{\Sigma}$, $\Phi=Id \times \hat{\phi}$ where $\hat{\phi}:\hat{\Sigma}\rightarrow M$ is a compact stable minimal  immersion of codimension $1$,  and therefore $\Phi(\Sigma)=\mathbb{H}P^{\frac{m_{1}}{4}}\times \hat{\phi}(\hat{\Sigma})$.
     In particular, for $m_{2}=1$, $\Sigma=\mathbb{H}P^{\frac{m_{1}}{4}}$, $\hat{\phi}$ is a constant function, and $\Phi(\Sigma)=\mathbb{H}P^{\frac{m_{1}}{4}}\times \{q\}$, for $q\in M$.
\item If $d=2$,
     $\Sigma=\mathbb{H}P^{\frac{m_{1}}{4}}\times \hat{\Sigma}$, $\Phi=Id\times \hat{\phi}$ where $\hat{\phi}:\hat{\Sigma}\rightarrow M$ is a compact stable minimal  immersion of codimension $2$,  and therefore $\Phi(\Sigma)=\mathbb{H}P^{\frac{m_{1}}{4}}\times \hat{\phi}(\hat{\Sigma})$.
     In particular, for $m_{2}=1$, there are no compact stable minimal immersions of codimension $2$ in $\mathbb{H}P^{\frac{m_{1}}{4}}\times M$. And for $m_{2}=2$, $\Sigma=\mathbb{H}P^{\frac{m_{1}}{4}}$, $\hat{\phi}$ is a constant function, and $\Phi(\Sigma)=\mathbb{H}P^{\frac{m_{1}}{4}}\times \{q\}$, for $q\in M$.
    \item If $n=1$, $\phi:\Sigma \rightarrow M$ is a  stable geodesic, $\psi$ is a constant function,  and therefore $\Phi(\Sigma)=\{r\}\times \phi(\Sigma)$ with $r$ a point of $\mathbb{H}P^{\frac{m_{1}}{4}}$.
    
    \item If $n=2$, $\phi:\Sigma \rightarrow M$ is a stable minimal immersion of dimension $2$, $\psi$ is a constant function, and therefore $\Phi(\Sigma)=\{r\}\times \phi(\Sigma)$ with $r$ a point of $\mathbb{H}P^{\frac{m_{1}}{4}}$.
\end{itemize}
\end{theorem}

\begin{remark}
As a consequence of the previous theorem, we have a complete classification of CSMI in $\mathbb{H}P^{1}\times M^{1}$. 
\end{remark}

As in the complex case, the last theorem tells us that the CSMI of codimension $1$ and $2$ or dimension $1$ and $2$ in $\mathbb{H}P^{\frac{m_{1}}{4}}\times M$ are the product of trivial CSMI of $\mathbb{H}P^{\frac{m_{1}}{4}}$  with CSMI of $M$. Therefore, as an application, using Theorems \ref{ohnita0} and \ref{ohnita}, we have:

\begin{corollary}
\label{A4}
There are no compact stable minimal immersions of 

\begin{itemize}
\item codimension $d=1$ in the product manifold  $\mathbb{H}P^{\frac{m_{1}}{4}}\times S^{s}$,  $\mathbb{H}P^{\frac{m_{1}}{4}}\times \mathbb{O} P^{2}$, or  $\mathbb{H}P^{\frac{m_{1}}{4}}\times \mathbb{K}P^{s}$, other than $\mathbb{H}P^{\frac{m_{1}}{4}}\times \{q\}$ in $\mathbb{H}P^{\frac{m_{1}}{4}}\times S^{1}$, for $q\in S^{1}$.

\item codimension $d=2$ in the product manifold $\mathbb{H}P^{\frac{m_{1}}{4}}\times S^{s}$, $\mathbb{H}P^{\frac{m_{1}}{4}}\times \mathbb{O} P^{2}$, or $\mathbb{H}P^{\frac{m_{1}}{4}}\times \mathbb{H}P^{s}$, other than $\mathbb{H}P^{\frac{m_{1}}{4}}\times \{q\}$ in  $\mathbb{H}P^{\frac{m_{1}}{4}}\times S^{2}$, for $q\in S^{2}$.

\item dimension $n=1$ in the product manifold  $\mathbb{H}P^{\frac{m_{1}}{4}}\times S^{s}$,  $\mathbb{H}P^{\frac{m_{1}}{4}}\times \mathbb{O} P^{2}$, or  $\mathbb{H}P^{\frac{m_{1}}{4}}\times \mathbb{K}P^{s}$, other than $\{r\}\times S^{1}$ in $\mathbb{H}P^{\frac{m_{1}}{4}}\times S^{1}$, for $r\in \mathbb{H}P^{\frac{m_{1}}{4}}$.

\item dimension $n=2$ in the product manifold $\mathbb{H}P^{\frac{m_{1}}{4}}\times S^{s}$, $\mathbb{H}P^{\frac{m_{1}}{4}}\times \mathbb{O} P^{2}$, or $\mathbb{H}P^{\frac{m_{1}}{4}}\times \mathbb{H}P^{s}$, other than $\{r\}\times S^{2}$ in $\mathbb{H}P^{\frac{m_{1}}{4}}\times S^{2}$, for $r\in \mathbb{H}P^{\frac{m_{1}}{4}}$.
\end{itemize}
Here, $\mathbb{K}\in \{\mathbb{C}, \mathbb{H}\}$.
\end{corollary}

\begin{remark}
The particular cases in Corollary \ref{A4} of CSMI of codimension $d=1$ in  $\mathbb{H}P^{\frac{m_{1}}{4}}\times S^{s}$ for $s\geq 2$, or the case of codimension $2$ or dimension $2$ in $\mathbb{H}P^{\frac{m_{1}}{4}}\times S^{s}$, for $s\geq 3$, can be also obtained as a consequence of Theorem \ref{TU} from  Torralbo and Urbano by setting $M=\mathbb{H}P^{\frac{m_{1}}{4}}$ and applying Theorem \ref{ohnita}.
\end{remark}

\begin{corollary} 
\label{A5}
The only compact stable minimal immersion of \begin{itemize}
    \item codimension $d=1$ in the product space $\mathbb{H}P^{\frac{m_{1}}{4}}\times \mathbb{R}P^{s}$ is $\mathbb{H}P^{\frac{m_{1}}{4}}\times \mathbb{R}P^{s-1}$.
    
    \item  codimension $d=2$ in the product space $\mathbb{H}P^{\frac{m_{1}}{4}}\times \mathbb{R}P^{s}$ is $\mathbb{H}P^{\frac{m_{1}}{4}}\times \mathbb{R}P^{s-2}$, and in $\mathbb{H}P^{\frac{m_{1}}{4}}\times \mathbb{C}P^{s}$ is $\mathbb{H}P^{\frac{m_{1}}{4}}\times M$, where $M$ is a complex submanifold of dimension $2s-2$ immersed in  $\mathbb{C}P^{s}$ .
    
    \item dimension $n=1$ in the product space $\mathbb{H}P^{\frac{m_{1}}{4}}\times \mathbb{R}P^{s}$ is $\{r\}\times \mathbb{R}P^{1}$, $r\in \mathbb{H}P^{\frac{m_{1}}{4}}$.
    
      \item dimension $n=2$ in the product space $\mathbb{H}P^{\frac{m_{1}}{4}}\times \mathbb{R}P^{s}$ is  $ \{r\}\times \mathbb{R}P^{2}$, and in  $\mathbb{H}P^{\frac{m_{1}}{4}}\times \mathbb{C}P^{s}$ is $\{r\}\times M$,  where $M$ is a complex submanifold of dimension $2$ immersed in  $\mathbb{C}P^{s}$  and $r\in \mathbb{H}P^{\frac{m_{1}}{4}}$.
    
\end{itemize}
\end{corollary}
This paper is structured as follows:\\
In the second section, we state important notations, formulas, and theorems needed for the developments in the third and fourth section. The second section is divided in three subsections.

\begin{itemize}
\item In Subsection \ref{subsetion0}, we define the Jacobi operator and the Morse index.  
\item In Subsection \ref{notations}, we set up notation.
    \item In Subsection \ref{section1}, we state some formulas involving the geometry of the complex and quaternionic projective spaces which are needed for computations presented in Subsections \ref{gfc} and \ref{quaternioniccompu}.
    \item In Subsection \ref{RS}, we recall some important definitions and propositions related to Riemannian submersions. We also establish Lemma \ref{lemmafact} which states a sufficient condition for a Riemannian submersion to be trivial. This lemma will be used in the proof of classification theorems of CSMI presented in Subsections \ref{CC1} and \ref{QC12}.
\end{itemize}
The third section is dedicated to the study of CSMI in the product of a complex projective space with any other Riemannian manifold. The third section is divided in four subsections.

\begin{itemize}
    \item In Subsection \ref{gfc}, we prove a general formula  that will be used throughout the current section.
    \item In Subsection \ref{CC1}, we prove a classification theorem where the codimension of the immersion is $1$. Moreover, we obtain some corollaries when the second manifold is a compact rank one space.
    \item In Subsection \ref{CC2}, we use the general formula obtained in Subsection \ref{gfc} for when the codimension or dimension of the immersion is two. This allows us to prove a characterization of CSMI in the product of two complex projective spaces. 
    \item In Subsection \ref{CD2}, we prove a classification theorem where the dimension of the immersion is $1$. Additionally, we obtain some corollaries when the second manifold is a compact rank one space.
\end{itemize}
The fourth section is dedicated to the study of CSMI in the product of a quaternionic projective space with any other Riemannian manifold. The fourth section is divided in three subsections.

\begin{itemize}
    \item In Subsection \ref{quaternioniccompu}, we prove a general formula  that will be used throughout this section. 
    \item In Subsection \ref{QC12}, we prove a classification theorem where the codimension of the immersion is $1$ and $2$. Moreover, we obtain some corollaries for when the second manifold is a compact rank one space. 
    
    \item Analogously, in Subsection \ref{QD12}, we prove a classification theorem where the dimension of the immersion is $1$ and $2$. Additionally,  we obtain some corollaries for when the second manifold is a compact rank one space.\\
    \end{itemize}
    
\textbf{Acknowledgments}
This paper was made possible thanks to a PhD scholarship (IMU Breakout Graduate Fellowship) from IMU and TWAS to the author. I  am very grateful for the patience and guidance of Professors Gonzalo Garc\'ia,  Fernando Marques,  and Heber Mesa.  Part of this work was done while the author was visiting Princeton University as a VSRC. I am grateful to Princeton University for the hospitality. I am also thankful to the Department of Mathematics at Universidad del Valle for partial support my visit to Princeton. Finally, I thank Shuli Chen for comments on this work.
\end{section}

\section{PRELIMINARIES}

\begin{subsection}{JACOBI OPERATOR}
\label{subsetion0}

Let $\Phi:\Sigma\rightarrow M$ be a compact Riemannian immersion, where $\Sigma$ and $M$ are Riemannian manifolds of dimensions $n$ and $n+d$, respectively. Let $F:\Sigma\times (-\epsilon,\epsilon)\rightarrow M$ be a smooth map such that $F(\cdot,0)=\Phi(\cdot)$. We denote $F_{t}(x):=F(x,t)$ and $\Sigma_{t}:=F_{t}(\Sigma)$. Then, we get the first variational formula 

\begin{theorem}{First variational formula:}
\begin{center}

    $\displaystyle \frac{d}{dt} |\Sigma_{t}|\biggr\vert_{t=0}=-\int_{\Sigma}\langle X,H \rangle d\Sigma$,
\end{center}
where $|\Sigma_{t}|$ denotes the area of $\Sigma_{t}$, $H$ is the mean curvature vector of $\Sigma$,   $W(p)=\frac{\partial F}{\partial t}(p,0)$ and $X:=W^{N}$. 
\end{theorem}
\begin{definition}
 We say that $\Sigma$ is minimal if $H=0$.
\end{definition}
\begin{remark}
Notice that $X\in \Gamma(N\Sigma)$, i.e., it is a section in the normal bundle of $\Sigma$.
\end{remark}

When studying minimal immersions, it is natural to ask about the extent to which they locally minimize area. This leads us to consider the second variation of the area functional.

\begin{theorem}{Second Variation formula.}\\ If $\Phi:\Sigma\rightarrow M$ is a compact minimal immersion, then: 
\begin{center}
    $\displaystyle \frac{d^{2}}{dt^{2}} |\Sigma_{t}|\biggr\vert_{t = 0}=-\int_{\Sigma} \langle J_{\Sigma}X,X \rangle d\Sigma$,
\end{center}
where $J_{\Sigma}$ is the elliptic Jacobi operator defined by\\

\begin{center}
    $\displaystyle J_{\Sigma}(X):=\triangle^{\perp}X+(\sum_{i=1}^{n}R^{M}(X,e_{i})e_{i})^{\perp}+\sum_{i,j=1}^{n} \langle B(e_{i},e_{j}),X \rangle B(e_{i},e_{j}) $,
\end{center}
and the normal Laplacian is given by

\begin{center}
    $\displaystyle \triangle^{\perp}X=\sum_{i=1}^{n}(\nabla^{\perp}_{e_{i}}\nabla^{\perp}_{e_{i}}X-\nabla^{\perp}_{(\nabla_{e_{i}}e_{i})^{T}}X)$.
\end{center}
Here, $\{e_{1},\ldots,e_{n}\}$ is an orthonormal basis of $T\Sigma$, $\nabla$ is the connection of $M$, $\nabla^{\perp}$ is the normal connection of $\Sigma$ in $M$, $B$ is the second fundamental form of $\Sigma$ in $M$, and $R^{M}$ is the curvature tensor of $M$. 
\end{theorem}

\begin{definition}
The Morse index of $\Sigma$ is the number of negative eigenvalues of $J_{\Sigma}$ counting multiplicities. We say that $\Sigma$ is stable if it has Morse index $0$, i.e.

\begin{center}
  $\displaystyle-\int_{\Sigma} \langle J_{\Sigma}X,X \rangle d\Sigma \geq 0$, for all $X\in \Gamma(N(M))$. 
\end{center}
\end{definition}
\begin{remark}
The Morse index gives us information about the number of directions in which our submanifold fails to minimize area. \\

\end{remark}

\end{subsection}

\begin{subsection}{NOTATION}
\label{notations}
Let $\bar{M}:=M_{1} \times M_{2}$ and
$p=(p_{1},p_{2})\in \bar{M}$ where $p_{i}$ is a point in $M_{i}$ for $i=1,2$
, and $M_{1}$ and $M_{2}$ are $m_{1}$ and $m_{2}$-dimensional Riemannian manifolds,  respectively. Then, we have the splitting

\begin{center}
$T_{p}(\bar{M})=T_{p_{1}}(M_{1}) \oplus T_{p_{2}}(M_{2})$    
\end{center}
i.e., if $x \in T_{p}(\bar{M})$, we have
    
    \begin{center}
        $x=(x^{1},x^{2})$
    \end{center}
where $x^{1}=P^{1}(x)\in T_{p_{1}}(M_{1})$, $x^{2}=P^{2}(x)\in T_{p_{2}}(M_{2})$ and $P^{1}$ and $P^{2}$ are the projections on $T_{p_{1}}(M_{1})$ and $T_{p_{2}}(M_{2})$ respectively.\\

Let us denote by $\bar{\nabla}, \nabla^{1}$ and  $\nabla^{2}$ the Riemannian connections of $\bar{M}, M_{1}$ and  $M_{2}$  respectively. Then, for $X, Y \in \chi(\bar{M})$,
\begin{equation}\label{SCM1M2}
    \bar{\nabla}_{X}Y(p)= (\nabla^{1}_{X^{1}}Y^{1}(p), \nabla^{2}_{X^{2}}Y^{2}(p)).
    \end{equation}
\end{subsection}

\begin{subsection}{PROJECTIVE SPACES}
\label{section1}

Let $\mathbb{C} P^{\frac{m_{1}}{2}}$ be the complex projective space of real dimension $m_{1}$ with the Fubini-Study metric, and
$\mathbb{H} P^{\frac{m_{1}}{4}}$ be the quaternionic projective space of real dimension $m_{1}$ with the standard metric. Let us consider the composition $\Phi_{1}:=i\circ \phi_{1}$, where $i$ is the inclusion map of $S^{l_{d}}$ in $\mathbb{R}^{{l_{d}}+1}$   and $\phi_{1}$ is the generalized Veronese imbedding 
\begin{center}
$\phi_{1}: \mathbb{K} P^{\frac{m_{1}}{d}} \rightarrow S^{l_{d}}$.
\end{center}
Here, $l_{d}=\frac{m_{1}}{2}(\frac{m_{1}}{d}+1)+\frac{m_{1}}{d}-1$, $\mathbb{K}\in \{\mathbb{C},\mathbb{H}\}$, and $d=dim_{\mathbb{R}}(\mathbb{K})$ (see Section 2 in \cite{sakamoto1977planar}). We will denote $m=l_{d}+1$.\\

The  constant holomorphic sectional curvature of $\mathbb{C} P^{\frac{m_{1}}{2}}$  is given by $\lambda^{2}:=\frac{m_{1}}{\frac{m_{1}}{2}+1}$. Moreover, if $R$ is the curvature tensor of $\mathbb{C} P^{\frac{m_{1}}{2}}$, then
\begin{equation}
\begin{aligned}
\label{fubini}
    \displaystyle \langle R(X,Y)Z,W\rangle = &\frac{\lambda^{2}}{4} \Big \{\langle Y,Z\rangle \langle X,W\rangle-\langle X,Z\rangle \langle Y,W \rangle \\
  & +\langle JY,Z \rangle \langle JX,W \rangle- \langle JX,Z \rangle \langle JY,W\rangle+2 \langle X,JY \rangle \langle JZ,W \rangle \Big \},
    \end{aligned}
\end{equation}
 where $J$ is the complex structure of $\mathbb{C} P^{\frac{m_{1}}{2}}$ (see Equation (1.1) in \cite{ohnita1986stable}). Notice that from Equation (\ref{fubini}), $\lambda^{2}$ is also the maximum of the sectional curvatures on $\mathbb{C} P^{\frac{m_{1}}{2}}$.\\

The maximum of the sectional curvatures on $\mathbb{H} P^{\frac{m_{1}}{2}}$  is given by $\lambda^{2}:=\frac{2m_{1}}{m_{1}+4}$. Moreover, if $R$ is the curvature tensor of $\mathbb{H} P^{\frac{m_{1}}{2}}$, then
 \begin{equation}
    \begin{aligned}
\label{hfubini}
\displaystyle \langle R(X,Y)Z,W\rangle =&\frac{\lambda^{2}}{4} \Big \{\langle Y,Z \rangle \langle X,W \rangle-\langle X,Z \rangle \langle Y,W \rangle  +2\sum_{k=1}^{3}\langle X,J_{k}(Y)\rangle \langle J_{k}(Z),W\rangle \\
    \displaystyle &+\sum_{k=1}^{3}(\langle J_{k}(Y),Z\rangle \langle J_{k}(X),W\rangle-\langle J_{k}(X),Z\rangle \langle J_{k}(Y),W \rangle)\Big \},
    \end{aligned}
\end{equation}
 where $J_{k}$, $k=1,2,3$, is a canonical local basis of the quaternionic Kaehler structure of $\mathbb{H}P^{\frac{m_{1}}{2}}$ (see Equation (1.2) in \cite{ohnita1986stable}).\\

Let $B$ be the second fundamental form of $\Phi_{1}$ (for both cases $\mathbb{K}=\mathbb{C}, \mathbb{H}$). Then, 
\begin{equation}
\begin{aligned}
\label{CCP}
   \displaystyle 3\langle B(X,Y),B(Z,W)\rangle=&\langle R(X,Z)W,Y \rangle+\langle R(X,W)Z,Y\rangle +\lambda^{2}\langle X,Y\rangle \langle Z,W\rangle\\
   \displaystyle &+\lambda^{2} \langle X,W\rangle \langle Y,Z\rangle+\lambda^{2}\langle X,Z\rangle \langle W,Y\rangle,
   \end{aligned}
\end{equation}
where  $X,Y,Z,W \in T_{q}\mathbb{K} P^{\frac{m_{1}}{d}}$, $q\in \mathbb{K} P^{\frac{m_{1}}{d}}$ (see Equation (3.10) in \cite{ohnita1986stable}). 

\end{subsection}

\begin{subsection}{RIEMANNIAN SUBMERSIONS}
\label{RS}

In this subsection we prove an important lemma (Lemma \ref{lemmafact}) which is used in the proof of classification theorems of compact stable minimal immersions of codimension $1$
in $\mathbb{C}P^{\frac{m_{1}}{2}}\times M$ (Theorem \ref{cod1c})  and codimension $1$ and $2$ in $\mathbb{H}P^{\frac{m_{1}}{4}}\times M$ (Theorems \ref{cod1h} and \ref{hz1}). To prove Lemma \ref{lemmafact}, we need some definitions and some important theorems found in  O'Neill \cite{o1966fundamental} and Hermann \cite{hermann1960sufficient}:

\begin{proposition}\cite{hermann1960sufficient}
\label{hermann}
Let $\Pi:M\rightarrow B$ be an  onto Riemannian submersion. If $M$ is complete, so is $B$. In particular, if $\sigma:[a,b]\rightarrow B$ is a geodesic segment in $B$, then for each point  $m\in M$ with $\Pi(m)\in \sigma(a)$, there exists a unique horizontal lift $\sigma^{m}:[a,b]\rightarrow M$ of $\sigma$ such that,
\begin{center}
$\sigma^{m}(a)=m$ and $\sigma^{m}$ is also geodesic.     
\end{center}
Now let $\phi_{\sigma}:\Pi^{-1}(\sigma(a))\rightarrow \Pi^{-1}(\sigma(b))$  be the function given by $\phi_{\sigma}(m)=\sigma^{m}(b)$ (see Fig.\ref{cube}). This function $\phi_{\sigma}$ is a diffeomorphism. If the fibers of $\Pi$  are totally geodesic submanifolds, then $\phi_{\sigma}$ is an isometry and $\Pi$ is a fibre bundle.  
\end{proposition}

\begin{figure}
\centering
\includegraphics[scale=0.3]{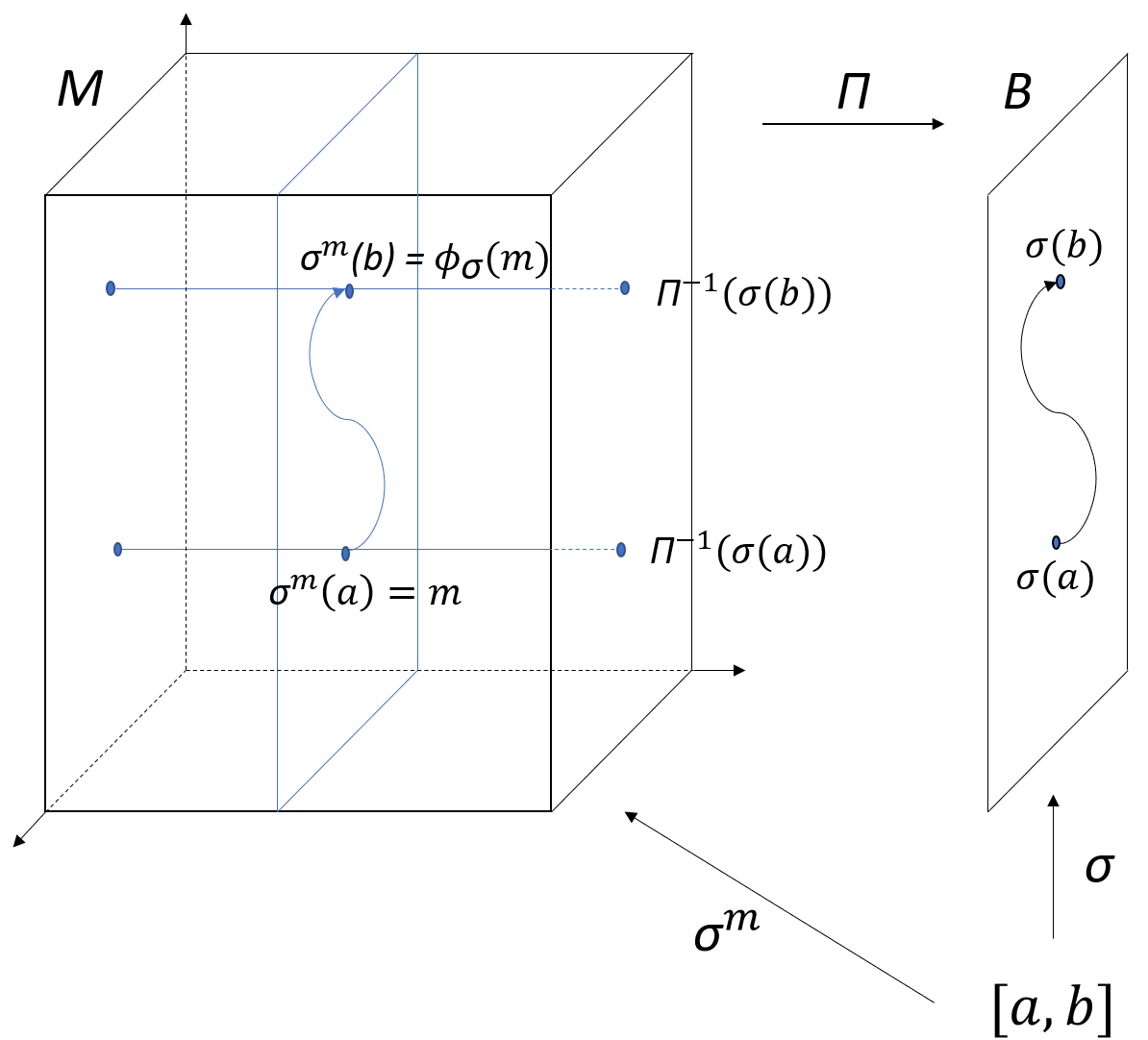}
\caption{\hspace{1cm}}
\label{cube}
\end{figure}

\begin{remark}
\label{lift}
Notice that since $\sigma^{m}$ is a lift of $\sigma$, then 
\begin{equation}
    \sigma(c)=\Pi(\sigma^{m}(c)),
\end{equation}
for all $c\in [a,b]$.
\end{remark}

\begin{definition} \cite{o1966fundamental}
Let $\Pi:M\rightarrow B$ be an  onto Riemannian submersion and $M$ complete. Fixed a point $o\in B$. The group of the submersion $\Pi$, $G_{\Pi}$, is given by

\begin{center}
    $G_{\Pi}:= \{\phi_{\sigma}:\Pi^{-1}(o)\rightarrow \Pi^{-1}(o); \sigma \text{\hspace{0.01cm} is a geodesic loop at \hspace{0.01cm}} o \}$,
\end{center}
with the composition of functions as the group operation.
\end{definition}

\begin{definition}
Let $\Pi:M\rightarrow B$ be a Riemmanian submersion. We say that $\Pi$ is a trivial submersion if $M=F\times B$, where $F$ is a Riemannian manifold, and $\Pi(f,b)=b$ is the projection of the Riemnnian product $F\times B$ onto the factor $B$, where $f\in F$ and $b\in B$.
\end{definition}
\begin{proposition}\cite{o1966fundamental}
\label{oneill}
Let $\Pi:M\rightarrow B$ be a onto Riemannian submersion of a complete Riemannian manifold $M$. Then, $\Pi$ is trivial if and only if the fibers of $\Pi$ are totally geodesic and the group of the submersion vanishes. 
\end{proposition}

Given the technicality of the proof of Lemmas \ref{lemmafact} and \ref{lemmafact2}, the reader may wish to skip them on a first reading and continue to Section \ref{complexcompu}.

\begin{lemma}
\label{lemmafact}
Let $M$ be a complete Riemannian manifold, B be a simply connected Riemannian manifold and $\Pi:M\rightarrow B $ an onto Riemannian submersion such that the fibers of $\Pi$ are totally geodesic and the horizontal distribution is integrable. Then, $\Pi$ is the trivial submersion.
\end{lemma}
\begin{proof}

Using Proposition \ref{oneill}, it is enough to prove that $G_{\Pi}$ is trivial. For a fixed element $o\in B$, let $\sigma:[a,b]\rightarrow B$ be a geodesic loop at $o$, i.e., $\sigma(a)=\sigma(b)=o$.  Then, we will prove that
\begin{center}

$\phi_{\sigma}:\Pi^{-1}(o) \rightarrow \Pi^{-1}(o)$ 
\end{center}
is the identity function. For $m\in \Pi^{-1}(\sigma(a))$, there exists a unique horizontal lift $\sigma^{m}:[a,b]\rightarrow M$ of $\sigma$ such that $\sigma^{m}(a)=m$. Moreover $\sigma^{m}$ is also geodesic.

 Since the horizontal distribution is integrable,  through $m\in M$, there passes a unique maximal connected integral manifold of the horizontal distribution, denoted by $(N,\varphi)$. Now we use the following lemma that will be proved at the end:

\begin{lemma}
\label{lemmafact2}
Under the conditions of Lemma \ref{lemmafact},
the restriction

\begin{center}
$\Pi|_{\varphi(N)}:(\varphi(N),\tau_{N})\rightarrow B$    
\end{center}
is a homeomorphism, where $\tau_{N}$ is the topology induced by $(N,\varphi)$, i.e., such that $\varphi:N\rightarrow (\varphi(N),\tau_{N})$ is a continuous map (notice that $\varphi:N\rightarrow (\varphi(N),\tau_{N})$ is an open map because $\varphi$ is 1-1).  

\end{lemma}

We have $\sigma^{m}([a,b])\subset \varphi(N)$ because  $\sigma^{m}$ is horizontal and $\sigma^{m}([a,b])$ is connected. 

Since $\sigma(a)=\sigma(b)$, we have that $\Pi(\sigma^{m}(a))=\Pi(\sigma^{m}(b))$, by  Remark \ref{lift}. Consequently,  $\Pi|_{\varphi(N)}(\sigma^{m}(a))=\Pi|_{\varphi(N)}(\sigma^{m}(b))$. 
By Lemma \ref{lemmafact2}, $\sigma^{m}(a)=\sigma^{m}(b)$ and thus,

\begin{center}
    $\phi_{\sigma}(m)=\sigma^{m}(b)=\sigma^{m}(a)=m.$
\end{center}
Since $m$ was arbitrary, $\phi_{\sigma}$ is the identity. This proves the  Lemma \ref{lemmafact}.\\

\end{proof}
Now, let us prove \textbf{ Lemma \ref{lemmafact2}}:
\begin{proof}
 Let us denote $\pi:=\Pi|_{\varphi(N)}$. First, we will prove that $\pi$ is a covering map of $B$. 
\begin{itemize}
    \item \textbf{$\pi$ is continuous}. Let $W$ be a open set in $B$. Since $\Pi\circ\varphi$ is continuous, 
    
    \begin{center}
        $(\Pi\circ\varphi)^{-1}(W)=\varphi^{-1}(\Pi^{-1}(W))=\varphi^{-1}(\varphi(N)\cap \Pi^{-1}(W))$\\
    \end{center}
    is an open set in $N$. 
    Therefore, $\varphi(N)\cap \Pi^{-1}(W)=\pi^{-1}(W)$ is open in $(\varphi(N),\tau_{N})$. 
    
    \item \textbf{$\pi$ is onto}.
    Let $b\in B$, $n\in \varphi (N)$ and  $q:=\Pi(n)$. Since $B$ is arcwise connected, there is a geodesic segment $\gamma: [0,1]\rightarrow B$ with $\gamma(0)=q$ and $\gamma(1)=b$. Consequently, there exists a unique horizontal lift $\gamma^{n}:[0,1]\rightarrow M$ of $\gamma$ such that 
    
    \begin{center}
        $\gamma^{n}(0)=n$\\
        $\gamma^{n}(1)\in \Pi^{-1}(b)$.\\
    \end{center}
    We have that $\gamma ^{n}([0,1])\subset \varphi(N)$, because $\gamma^{n}$ is horizontal, $\gamma^{n}(0)\in \varphi(N)$, and $\gamma^{n}([0,1])$ is connected.
    Therefore,  $b=\Pi(\gamma^{n}(1))=\pi(\gamma^{n}(1))$, for $\gamma^{n}(1)\in \varphi(N)$. Since $b$ was arbitrary, $\pi$ is onto.

    \item \textbf{Disjoint union.}  Let $x\in B$.
    By Proposition \ref{hermann} in \cite{hermann1960sufficient}, $\Pi$ is a fiber bundle. Since the horizonal distribution is integrable, there exists a connected open set $U$ of $B$ containing $x$, such that 
    
    \begin{equation}
    \Pi^{-1}(U)=U\times F  
    \end{equation}
    \begin{equation}
    \label{map}
        \Pi(u,f)=u,
    \end{equation}
for $(u,f)\in U\times F$, where $F$ is a typical fiber of $\Pi$ ($F=\Pi^{-1}(b)$, for some $b\in B$). Thus,\\

\begin{center}
    $\displaystyle \pi^{-1}(U)=\Pi^{-1}(U)\cap \varphi(N)=\bigcup_{f\in F}(U\times \{f\})\cap \varphi(N)$.
\end{center}
It is evident that the last union is disjoint.  Let us consider the following set,
\begin{center}
 $\displaystyle F':=\{f\in F: (U\times \{f\})\cap \varphi(N)\neq \emptyset \}$.
 \end{center}
 Notice that $F'\neq \emptyset$, because $\pi$ is onto, and then $\pi^{-1}(U)\neq \emptyset$.
For $f\in F'$, we have that $U\times \{f\}\subset \varphi(N)$, because $U\times \{f\}$ is a connected integral manifold of the horizontal distribution, and $(U\times \{f\})\cap \varphi(N)\neq \emptyset$. Therefore,

\begin{center}
    $\displaystyle \pi^{-1}(U)=\bigcup_{f\in F'}U\times \{f\}$.
\end{center}

\item \textbf{The set $U\times \{f\}$, for $f\in F'$ is an open set of $(\varphi(N),\tau_{N})$}. Notice that $(U\times \{f\},\tau_{u})$ is a topological manifold with smooth structure such that $i:U\times \{f\}\rightarrow M$ is an smooth embedding, where $\tau_{u}$ is the subspace topology induced by $M$ (equivalently induced by $U\times F$). By Theorem 1.62 in \cite{warner2013foundations}, there exists a unique $C^{\infty}$ map $\bar{i}:(U\times\{f\},\tau_{u})\rightarrow N$ such that $\varphi \circ \bar{i} =i$, i.e. the diagram in Fig. \ref{diagram}  commutes. 
\begin{figure}
\begin{tikzcd}
(U\times \{f\},\tau_{u}) \arrow[rd, "\bar{i}"] \arrow[r, "i"] & M \\
& N \arrow[u, "\varphi"]
\end{tikzcd}
\caption{Commutative diagram.}
\label{diagram}
\end{figure}

The map $\bar{i}:(U\times \{f\},\tau_{u})\rightarrow N$ is nonsingular, because $i:U\times\{f\}\rightarrow M$ is nonsingular. Applying the inverse function theorem (see Proposition 5.16 in \cite{lee2013smooth}) to the $C^{\infty}$ function $\bar{i}:(U\times\{f\},\tau_{u})\rightarrow N$, we have that $\bar{i}$ is a local diffeomorphism. Therefore, $\bar{i}$ is an open map, and then 
$\bar{i}(U\times \{f\})$ is an open set of $N$. Since $\varphi:N\rightarrow (\varphi(N),\tau_{N})$ is an open map, it follows that $\varphi(\bar{i}(U\times \{f\}))=i(U\times \{f\})=U\times \{f\}$ is open in ($\varphi(N)$,$\tau_{N}$).  

\item \textbf{The mapping, 
\begin{equation}
\label{homeo}
\displaystyle \pi|_{U\times\{f\}}:(U\times\{f\},\tau_{\varphi(N)})\rightarrow B 
\end{equation}
is a homeomorphism, where $\tau_{\varphi(N)}$ is the subspace topology in $U\times\{f\}$  induced by $(\varphi(N),\tau_{N})$}. Notice that from Equation (\ref{map}) the mapping,
\begin{center}
     $\displaystyle \pi|_{U\times\{f\}}:(U\times\{f\},\tau_{u})\rightarrow U$
\end{center}
is an homeomorphism. Therefore, it is enough to prove that $\tau_{u}=\tau_{\varphi(N)}$.\\

Let $W\in \tau_{u}$. The set $\varphi(\bar{i}(W))=i(W)=W$ is open in $(\varphi(N),\tau_{N})$, because $\bar{i}$ and $\varphi$ are open maps. Since $W=W\cap (U\times \{f\})$, we have that $W\in \tau_{\varphi(N)}$. \\

Now, let $W\in \tau_{\varphi(N)}$. There exists an open set $\omega$ in $(\varphi(N),\tau_{N})$ such that $W=(U\times \{f\})\cap \omega$.
Since $U\times \{f\}$ is open in $(\varphi(N),\tau_{N})$, it follows that $W$ is an open set of $(\varphi(N),\tau_{N})$, and therefore $\varphi^{-1}(W)$ is open in $N$. 
Thus, $\bar{i}^{-1}(\varphi^{-1}(W))\in \tau_{u}$, because $\bar{i}$ is continuous. From the following equality,
\begin{center}
$\displaystyle \bar{i}^{-1}(\varphi^{-1}(W))=(\varphi\circ\bar{i})^{-1}(W)=i^{-1}(W)=W$,
\end{center}
we conclude that $W\in \tau_{u}$.
\end{itemize}
 Until now, we have proved that $\Pi|_{\varphi(N)}:(\varphi(N),\tau_{N})\rightarrow B$ is a covering map of $B$. Since $B$ is simply connected and applying exercise 6.1, Chapter 5 in  \cite{massey1967algebraic}, we have that the map $\Pi|_{\varphi(N)}:(\varphi(N),\tau_{N})\rightarrow B$ is an homeomorphism.
\end{proof}

\end{subsection}

\section{MINIMAL STABLE SUBMANIFOLDS IN $\mathbb{C}P^{\frac{m_{1}}{2}}\times M$}
\label{complexcompu}

\begin{subsection}{GENERAL FORMULA}
\label{gfc}
In this subsection we prove Lemma \ref{GENERALC}, which is fundamental for the development of main theorems in Subsections \ref{CC1}, \ref{CC2}, \ref{CD2}.\\

Let $\Phi=(\psi,\phi):\Sigma \rightarrow \bar{M}:=\mathbb{C}P^{\frac{m_{1}}{2}}\times M$ be a  compact minimal immersion of codimension $d$ and dimension $n$, where $M$ is any Riemannian manifold of dimension $m_{2}$ and $\Phi_{1}:\mathbb{C}P^{\frac{m_{1}}{2}}\rightarrow \mathbb{R}^{m}$ is the immersion described in Section \ref{section1}. For each $v\in \mathbb{R}^{m}$ let us consider the following:
\begin{center}
    $\nu:=(v,0)\in T(\mathbb{R}^{m}\times M)$\\
    $N_{v}:=[\nu]^{N}$,
\end{center}
where $[.]^{N}$ is  projection in the orthogonal complement, $N_{p}\Sigma$, of $T_{p}\Sigma$ in $T_{\Phi(p)}\bar{M}$, $p\in \Sigma$.

\begin{lemma}
\label{GENERALC}
Let $p\in \Sigma$, $\{e_{1},\ldots,e_{n}\}$
be an orthonormal basis of  $T_{p}\Sigma$, $\{\eta_{1},\ldots,\eta_{d}\}$ be an orthonormal basis of  $N_{p}\Sigma$, and $\{E_{1},\ldots,E_{m}\}$ be the usual canonical basis of $\mathbb{R}^{m}$. Then,
\begin{equation}
\label{MAIN}
\displaystyle\sum_{A=1}^{m} -\langle N_{E_{A}},J_{\Sigma}(N_{E_{A}})\rangle=\lambda^{2}\Big (\sum_{k=1}^{d}\sum_{l=1}^{d}\langle J(\eta^{1}_{k}),\eta^{1}_{l}\rangle^{2}-\langle \eta^{1}_{k},\eta^{1}_{l}\rangle^{2}\Big )
\end{equation}
\begin{equation}
    \label{MAIN2}\hspace{3.6cm}=\lambda^{2} \Big (\sum_{j=1}^{n}\sum_{i=1}^{n}
\langle J(e^{1}_{j}),e_{i}^{1}\rangle^{2}-\langle e^{1}_{j},e^{1}_{i}\rangle^{2}\Big ).
\end{equation}
\end{lemma}

\begin{proof}
Recall that $R$ is the curvature tensor of $\mathbb{C}P^{\frac{m_{1}}{2}}$ and $B$ is the second fundamental form of $\mathbb{C}P^{\frac{m_{1}}{2}}$ in $\mathbb{R}^{m}$ (see Subsection 1.2).
 By Equation (2.8) in \cite{chen2013stable}, we have the following

\begin{equation}
\label{mainequation}
    \sum_{A=1}^{m} -\langle N_{E_{A}},J_{\Sigma}(N_{E_{A}})\rangle=\sum_{j=1}^{n} \sum_{k=1}^{d}2|B(e_{j}^{1},\eta_{k}^{1})|^{2}-\langle B(\eta_{k}^{1},\eta_{k}^{1}),B(e_{j}^{1},e_{j}^{1})\rangle.
\end{equation}
Using Equation (\ref{CCP}), we have:

\begin{center}
    $\displaystyle3|B(e^{1}_{j},\eta^{1}_{k})|^{2}=\langle R(e^{1}_{j},\eta^{1}_{k})e^{1}_{j},\eta^{1}_{k}\rangle+ 2\lambda^{2}\langle e^{1}_{j},\eta^{1}_{k}\rangle^{2}+\lambda^{2}|e^{1}_{j}|^{2}|\eta^{1}_{k}|^{2}
    $
\end{center}
and 
\begin{center}
    $\displaystyle 3\langle B(\eta_{k}^{1}, \eta_{k}^{1}),B(e_{j}^{1},e_{j}^{1})\rangle=
    -2\langle R(e^{1}_{j},\eta^{1}_{k})e^{1}_{j},\eta^{1}_{k}\rangle  +\lambda^{2} |e^{1}_{j}|^{2}|\eta^{1}_{k}|^{2}+2\lambda^{2}\langle e^{1}_{j},\eta^{1}_{k}\rangle^{2} $.
\end{center}
Now using the last two equalities in Equation (\ref{mainequation}),

\begin{center}
   $\displaystyle\sum_{A=1}^{m} -\langle N_{E_{A}},J_{\Sigma}(N_{E_{A}})\rangle$\\
  $\displaystyle=\sum_{j=1}^{n}\sum_{k=1}^{d} \frac{2}{3} \Big (\langle R(e^{1}_{j},\eta^{1}_{k})e^{1}_{j},\eta^{1}_{k}\rangle+ 2\lambda^{2}\langle e^{1}_{j},\eta^{1}_{k}\rangle^{2}+\lambda^{2}|e^{1}_{j}|^{2}|\eta^{1}_{k}|^{2}\Big )$\\
 $\displaystyle -\frac{1}{3} \Big ( -2\langle R(e^{1}_{j},\eta^{1}_{k})e^{1}_{j},\eta^{1}_{k}\rangle  +\lambda^{2} |e^{1}_{j}|^{2}|\eta^{1}_{k}|^{2}+2\lambda^{2}\langle e^{1}_{j},\eta^{1}_{k}\rangle^{2} \Big )$\\
   
\end{center}
\begin{equation}
\label{baseequation}
    \displaystyle=\sum_{j=1}^{n}\sum_{k=1}^{d}-\frac{4}{3} \langle R(e^{1}_{j},\eta^{1}_{k})\eta^{1}_{k},e^{1}_{j}\rangle+\frac{2\lambda^{2}}{3}\langle e^{1}_{j},\eta^{1}_{k}\rangle^{2}+\frac{\lambda^{2}}{3}|e^{1}_{j}|^{2}|\eta^{1}_{k}|^{2}.
\end{equation}
Using Equation (\ref{fubini}) in Equation (\ref{baseequation}),

\begin{equation}
\label{L0}
\sum_{A=1}^{m} -\langle N_{E_{A}},J_{\Sigma}(N_{E_{A}})\rangle=\lambda^{2} \Big (\sum_{j=1}^{n}\sum_{k=1}^{d}\langle e^{1}_{j},\eta^{1}_{k}\rangle^{2}-\langle e^{1}_{j},J(\eta^{1}_{k})\rangle^{2} \Big ).
\end{equation}\\\

For $k\in\{1,\ldots,d\}$,

\begin{equation*}
\begin{aligned}
\displaystyle |\eta^{1}_{k}|^{2}&=|J(\eta^{1}_{k})|^{2}=|(J(\eta^{1}_{k}),0)|^{2}\\
\displaystyle &=\sum_{j=1}^{n}\langle (J(\eta^{1}_{k}),0),e_{j}\rangle^{2}+\sum_{l=1}^{d}\langle(J(\eta^{1}_{k}),0),\eta_{l}\rangle^{2}\\
\displaystyle &=\sum_{j=1}^{n}\langle J(\eta^{1}_{k}),e_{j}^{1}\rangle^{2}+\sum_{l=1}^{d}\langle J(\eta^{1}_{k}),\eta^{1}_{l}\rangle^{2}.
\end{aligned}
\end{equation*}
Then,

\begin{center}
$\displaystyle -\sum_{j=1}^{n}\langle J(\eta^{1}_{k}),e_{j}^{1}\rangle^{2}=-|\eta^{1}_{k}|^{2}+\sum_{l=1}^{d}\langle J(\eta^{1}_{k}),\eta^{1}_{l}\rangle^{2}$,
\end{center}
and summing in $k$,

\begin{equation}
\label{L1}
\displaystyle -\sum_{j=1}^{n}\sum_{k=1}^{d}\langle J(\eta^{1}_{k}),e_{j}^{1}\rangle^{2}=-\sum_{k=1}^{d}|\eta^{1}_{k}|^{2}+\sum_{k=1}^{d}\sum_{l=1}^{d}\langle J(\eta^{1}_{k}),\eta^{1}_{l}\rangle^{2}.
\end{equation}
On the other hand, again for $k\in \{1,\ldots,d\}$

\begin{equation*}
\begin{aligned}
\displaystyle |\eta^{1}_{k}|^{2}&=|(\eta^{1}_{k},0)|^{2}\\
\displaystyle &=\sum_{j=1}^{n}\langle(\eta^{1}_{k},0),e_{j}\rangle^{2}+\sum_{l=1}^{d}\langle(\eta^{1}_{k},0),\eta_{l}\rangle^{2}\\
\displaystyle &=\sum_{j=1}^{n}\langle \eta^{1}_{k},e_{j}^{1}\rangle^{2}+\sum_{l=1}^{d}\langle\eta^{1}_{k},\eta^{1}_{l}\rangle^{2}.
\end{aligned}
\end{equation*}
Therefore,

\begin{center}
$\displaystyle \sum_{j=1}^{n}\langle \eta^{1}_{k},e_{j}^{1}\rangle^{2} =|\eta^{1}_{k}|^{2}-\sum_{l=1}^{d}\langle \eta^{1}_{k},\eta^{1}_{l}\rangle^{2}$,
\end{center}
and summing in $k$,

\begin{equation}
\label{L2}
    \sum_{j=1}^{n}\sum_{k=1}^{d}\langle\eta^{1}_{k},e_{j}^{1}\rangle^{2} =\sum_{k=1}^{d}|\eta^{1}_{k}|^{2}-\sum_{k=1}^{d}\sum_{l=1}^{d}\langle\eta^{1}_{k},\eta^{1}_{l}\rangle^{2}.
\end{equation}
Then, in order to prove Equation (\ref{MAIN}), we replace Equations (\ref{L1}) and (\ref{L2}) in (\ref{L0}),

\begin{equation*}
\begin{aligned}
\displaystyle   \sum_{A=1}^{m} -\langle N_{E_{A}},J_{\Sigma}(N_{E_{A}})\rangle &=\lambda^{2} \Big ( \sum_{k=1}^{d}|\eta^{1}_{k}|^{2}-\sum_{k=1}^{d}\sum_{l=1}^{d}\langle \eta^{1}_{k},\eta^{1}_{l}\rangle^{2}-\sum_{k=1}^{d}|\eta^{1}_{k}|^{2}+\sum_{k=1}^{d}\sum_{l=1}^{d}\langle J(\eta^{1}_{k}),\eta^{1}_{l}\rangle^{2} \Big )\\
\displaystyle &=\lambda^{2}\Big (\sum_{k=1}^{d}\sum_{l=1}^{d}\langle J(\eta^{1}_{k}),\eta^{1}_{l}\rangle^{2}-\langle \eta^{1}_{k},\eta^{1}_{l}\rangle^{2}\Big ).\\
\end{aligned}
\end{equation*}

Now, let us prove Equation (\ref{MAIN2}). For $j\in\{1,\ldots,n\}$,

\begin{equation*}
\begin{aligned}
    \displaystyle |e^{1}_{j}|^{2}&=|J(e^{1}_{j})|^{2}=|(J(e^{1}_{j}),0)|^{2}\\
    \displaystyle &=\sum_{i=1}^{n}\langle(J(e^{1}_{j}),0),e_{i}\rangle^{2}+\sum_{k=1}^{d} \langle(J(e^{1}_{j}),0),\eta_{k}\rangle^{2}\\
    \displaystyle &=\sum_{i=1}^{n}\langle J(e^{1}_{j}),e^{1}_{i}\rangle^{2}+\sum_{k=1}^{d} \langle J(e^{1}_{j}),\eta_{k}^{1}\rangle^{2}.\\
\end{aligned}
\end{equation*}
Then,
\begin{center}
    $\displaystyle -\sum_{k=1}^{d} \langle J(e^{1}_{j}),\eta_{k}^{1}\rangle^{2}=-|e^{1}_{j}|^{2}+\sum_{i=1}^{n}\langle J(e^{1}_{j}),e_{i}^{1}\rangle^{2}$,
\end{center}
and summing in $j$,
\begin{equation}
\label{s1}
    \displaystyle -\sum_{j=1}^{n}\sum_{k=1}^{d} \langle J(e^{1}_{j}),\eta_{k}^{1}\rangle^{2}=-\sum_{j=1}^{n}|e^{1}_{j}|^{2}+\sum_{j=1}^{n}\sum_{i=1}^{n}\langle J(e^{1}_{j}),e_{i}^{1}\rangle^{2}.
\end{equation}
On the other hand, again for $j\in \{1,\ldots,n\}$

\begin{equation*}
    \begin{aligned}
    \displaystyle |e^{1}_{j}|^{2}&=|(e^{1}_{j},0)|^{2}\\
    \displaystyle &=\sum_{i=1}^{n}\langle(e^{1}_{j},0),e_{i}\rangle^{2}+\sum_{k=1}^{d} \langle(e^{1}_{j},0),\eta_{k}\rangle^{2}\\
    \displaystyle &=\sum_{i=1}^{n}\langle e^{1}_{j},e^{1}_{i}\rangle^{2}+\sum_{k=1}^{d} \langle e^{1}_{j},\eta_{k}^{1}\rangle^{2}.\\
\end{aligned}
\end{equation*}
Therefore,

\begin{center}
    $\displaystyle \sum_{k=1}^{d} \langle e^{1}_{j},\eta_{k}^{1}\rangle^{2}=|e^{1}_{j}|^{2}-\sum_{i=1}^{n}\langle e^{1}_{j},e^{1}_{i}\rangle^{2}$,\\
\end{center}
summing in $j$,

\begin{equation}
\label{s2}
    \displaystyle \sum_{j=1}^{n}\sum_{k=1}^{d} \langle e^{1}_{j},\eta_{k}^{1}\rangle^{2}=\sum_{j=1}^{n}|e^{1}_{j}|^{2}-\sum_{j=1}^{n}\sum_{i=1}^{n}\langle e^{1}_{j},e^{1}_{i}\rangle^{2}.
\end{equation}
Replacing Equations (\ref{s1}) and (\ref{s2}) in (\ref{L0})

\begin{equation*}
\begin{aligned}
\displaystyle
\sum_{A=1}^{m} -\langle N_{E_{A}},J_{\Sigma}(N_{E_{A}})\rangle &=\lambda^{2}\Big (\sum_{j=1}^{n}|e^{1}_{j}|^{2}-\sum_{j=1}^{n}\sum_{i=1}^{n}\langle e^{1}_{j},e^{1}_{i}\rangle^{2}
-\sum_{j=1}^{n}|e^{1}_{j}|^{2}+\sum_{j=1}^{n}\sum_{i=1}^{n}\langle J(e^{1}_{j}),e_{i}^{1}\rangle^{2})\Big )
\\
\displaystyle & =\lambda^{2} \Big ( \sum_{j=1}^{n}\sum_{i=1}^{n}
\langle J(e^{1}_{j}),e_{i}^{1}\rangle^{2}-\langle e^{1}_{j},e^{1}_{i}\rangle^{2}\Big ).
\end{aligned}
\end{equation*}
\end{proof}

\end{subsection}

\subsection{CODIMENSION $1$}
\label{CC1}

In this subsection, we will use the general formula proved in the last subsection to prove a classification theorem for compact stable minimal immersions of codimension $1$ in the product of a complex projective space with any other Riemannian manifold. Moreover, as an application, we obtain some corollaries when the second manifold is a compact rank one space.
\begin{theorem}
\label{cod1c}
Let $\Phi=(\psi,\phi):\Sigma \rightarrow \bar{M}:= \mathbb{C}P^{\frac{m_{1}}{2}}\times M$ be a compact stable minimal immersion of codimension $d=1$, where $M$ is any Riemannian manifold of dimension $m_{2}$. Then, $\Sigma=\mathbb{C}P^{\frac{m_{1}}{2}}\times \hat{\Sigma}$, $\Phi=Id\times \hat{\phi}$ where $\hat{\phi}:\hat{\Sigma}\rightarrow M$ is a compact stable minimal  immersion of codimension $1$, and therefore $\Phi(\Sigma)=\mathbb{C}P^{\frac{m_{1}}{2}}\times \hat{\phi}(\hat{\Sigma})$. In particular, for $m_{2}=1$, $\Sigma=\mathbb{C}P^{\frac{m_{1}}{2}}$, $\hat{\phi}$ is a constant function, and $\Phi(\Sigma)=\mathbb{C}P^{\frac{m_{1}}{2}}\times \{q\}$, for $q\in M$.
\end{theorem}

\begin{proof}
Since $d=1$, Equation (\ref{MAIN}) becomes

\begin{equation}
\label{similar}
\displaystyle\sum_{A=1}^{m} -\langle N_{E_{A}},J_{\Sigma}(N_{E_{A}})\rangle=\lambda^{2}(\langle J(\eta^{1}),\eta^{1}\rangle^{2}-\langle\eta^{1},\eta^{1}\rangle^{2})=-\lambda^{2}|\eta^{1}|^{4},
\end{equation}
where $\eta$ is an unitary vector in $N_{p}\Sigma$, for $p \in \Sigma$. Therefore,

\begin{center}
$\displaystyle 0\leq \sum_{A=1}^{m} -\int_{\Sigma}\langle N_{E_{A}},J_{\Sigma}(N_{E_{A}})\rangle d\Sigma=
-\lambda^{2}\int_{\Sigma}|\eta^{1}|^{4}d\Sigma \leq 0,$
\end{center}
where we have used the fact that $\Sigma$ is stable in the first inequality. Hence, for $p\in \Sigma$, $\eta^{1}=0$, and therefore $\eta=(0,\eta^{2})$. Then,

\begin{center}
    $ d\Phi_{p}(T_{p}\Sigma)= \bar{D_{1}}(p) \bigoplus \bar{D_{2}}(p)$,
\end{center}
where $\bar{D_{1}}$ and $\bar{D_{2}}$ are given by: 

\begin{center}
    $\bar{D_{1}}(p)=\{(x,0): x\in T_{\psi(p)}\mathbb{C}P^{\frac{m_{1}}{2}}\} $\\
    
    $\bar{D_{2}}(p)= \{(0,w): w\in [\eta^{2}]^{\perp_{M}}\}$,
\end{center}
where $[z]^{\perp_{M}}$ is the orthogonal complement of $z$ in $T_{\phi(p)}M$. Since $d\Phi_{p}(T_{p}\Sigma)$ is isometric to $T_{p}\Sigma$, $\bar{D_{1}}$ and $\bar{D_{2}}$ induce orthogonal complementary smooth distributions $D_{1}$ and $D_{2}$ on $\Sigma$ given by

\begin{center}
    $D_{1}(p)=\{h\in T_{p}\Sigma: d\Phi_{p}(h)\in \bar{D_{1}}(p)\}=\{h\in T_{p}\Sigma: d\phi_{p}(h)=0\}=\text{ker}(d\phi_{p})$\\
    
     $D_{2}(p)=\{v\in T_{p}\Sigma: d\Phi_{p}(v)\in \bar{D_{2}}(p)\}=\{v\in T_{p}\Sigma: d\psi_{p}(v)=0\}=\text{ker}(d\psi_{p})$.
\end{center}
\begin{lemma}
The function $\psi:\Sigma \rightarrow \mathbb{C}P^{\frac{m_{1}}{2}}$ is an onto Riemannian submersion, with horizontal and vertical distributions given by $D_{1}$ and $D_{2}$, respectively. Moreover, $D_{1}$ and $D_{2}$ are totally geodesic distributions.\\
\end{lemma}

\begin{proof}

\begin{itemize}
    \item \textbf{The mapping $d\psi_{p}:T_{p}\Sigma \rightarrow T_{\psi(p)}\mathbb{C}P^{\frac{m_{1}}{2}}$ is onto.} Let $x\in T_{\psi(p)}\mathbb{C}P^{\frac{m_{1}}{2}}$. Since $(x,0)\in \bar{D_{1}}(p)\subset  d\Phi_{p}(T_{p}\Sigma)$, there exists $h\in T_{p}\Sigma$ such that 
    \begin{center}
    $d\Phi_{p}(h)=(d\psi_{p}(h),d\phi_{p}(h))=(x,0)$.
    \end{center}
    Therefore, $d\psi_{p}(h)=x$.
    
\item By definition the vertical vectors $v$ of $\psi$ at $p\in \Sigma$ are such that $d\psi_{p}(v)=0$. Therefore, $D_{2}(p)$ consists of the vertical vectors, and thus $D_{1}(p)$ consists of the horizontal vectors.
    
\item \textbf{$d\psi_{p}$ preserves the length of horizontal vectors.} Let $h_{1},h_{2}$ be horizontal vectors. Therefore, $h_{1},h_{2}\in D_{1}(p)$ and then $d\phi_{p}(h_{1})=d\phi_{p}(h_{2})=0$. Since $\Phi$ is an isometric immersion,
\begin{center}
    $\langle d\Phi_{p}(h_{1}),d\Phi_{p}(h_{2})\rangle=\langle h_{1},h_{2}\rangle.$
\end{center}
But,
\begin{equation*}
\begin{aligned}
\displaystyle\langle d\Phi_{p}(h_{1}),d\Phi_{p}(h_{2})\rangle \displaystyle &=\langle(d\psi_{p}(h_{1}),d\phi_{p}(h_{1})),(d\psi_{p}(h_{2}),d\phi_{p}(h_{2}))\rangle\\
&=\langle d\psi_{p}(h_{1}),d\psi_{p}(h_{2})\rangle.
    \end{aligned}
    \end{equation*}
Then,
\begin{center}
    $\displaystyle \langle d\psi_{p}(h_{1}),d\psi_{p}(h_{2})\rangle= \langle h_{1},h_{2}\rangle.$
\end{center}
\item \textbf{$\psi$ is onto.} Until now, we have that $\psi$ is a Riemannian submersion.  By properties of submersions (Proposition 5.18 \cite{lee2013smooth}), $\psi(\Sigma)$ is an open set in $\mathbb{C}P^{\frac{m_{1}}{2}}$. Now $\psi:\Sigma \rightarrow \psi(\Sigma)$ is an onto Riemannian submersion. Since $\Sigma$ is complete, $\psi(\Sigma)$ is complete by Hermann \cite{hermann1960sufficient} and then closed. Therefore, $\psi(\Sigma)=\mathbb{C}P^{\frac{m_{1}}{2}} $ because 
$\psi(\Sigma)$  is a closed and open set of the connected set $\mathbb{C}P^{\frac{m_{1}}{2}}$.
\end{itemize}

\begin{itemize}
    \item \textbf{$D_{1}$ and $D_{2}$ are totally geodesic distributions.} Let $\nabla$ and $\bar{\nabla}$ be the connections of Levi-Civita on $\Sigma$ and $\bar{M}$, respectively. Let $q=(q_{1},q_{2})\in\bar{M}$ and $P:T_{q}\bar{M}\rightarrow T_{q}\bar{M}$ be a mapping given by
    
    \begin{center}
    $P(v_{1},v_{2})=(v_{1},-v_{2})$, where $v_{1}\in T_{q_{1}}\mathbb{C}P^{\frac{m_{1}}{2}}$ and $v_{2}\in T_{q_{2}}M$.\\
    \end{center}
    The map $P$ is a linear isometry that is parallel, $\bar{\nabla}P=0$, i.e. $(\bar{\nabla}_{A}P)C=0$, for all $A,C \in T_{q}\bar{M}$.\\
    
    Let us define at every point $p\in \Sigma$ the mapping $P^{\Sigma}:T_{p}\Sigma \rightarrow T_{p}\Sigma$ given by
    \begin{center}
    $P^{\Sigma}(w_{1}+w_{2})=w_{1}-w_{2}$, where $w_{1}\in D_{1}(p)$ and $w_{2}\in D_{2}(p)$. 
    \end{center}
    Notice that $P^{\Sigma}$ is a Riemannian almost product structure on $\Sigma$, where the eigenspaces of the eigenvalues $1$ and $-1$ of the operator $P^{\Sigma}$  are precisely given by $D_{1}(p)$ and $D_{2}(p)$, respectively. We have the following properties: \\
    
    \begin{enumerate}
        \item $d\Phi(P^{\Sigma})=P(d\Phi)$. Let $x=x_{1}+x_{2}\in T_{p}\Sigma$, where $x_{i}\in D_{i}(p)$, $i=1,2$. Then,
        
         \begin{equation*}
        \begin{aligned}
          d\Phi_{p}(P^{\Sigma}(x))&=d\Phi_{p}(x_{1}-x_{2})=d\Phi_{p}(x_{1})-d\Phi_{p}(x_{2})\\   &=(d\psi_{p}(x_{1}),0)-(0,d\phi_{p}(x_{2}))=(d\psi_{p}(x_{1}),-d\phi_{p}(x_{2}))\\
&=P(d\psi_{p}(x_{1}),d\phi_{p}(x_{2}))=P(d\psi_{p}(x_{1}+x_{2}),d\phi_{p}(x_{1}+x_{2}))\\
            &=P(d\Phi_{p}(x)).
            \end{aligned}
            \end{equation*}
        \item 
        $P([x]^{\Sigma})=[P(x)]^{\Sigma}$, where $[\cdot]^{\Sigma}$ is the projection onto $d\Phi(T\Sigma)$. Let $x\in T_{\Phi(p)}\bar{M}$, and $\{e_{i}\}_{i=1}^{n}$ be an orthonormal basis of $T_{p}\Sigma$. Then,
    
    \begin{equation*}
    \begin{aligned}\displaystyle P([x]^{\Sigma})&=P(\sum_{i=1}^{n}\langle x,d\Phi_{p}(e_{i})\rangle d\Phi_{p}(e_{i}))\\
        &=\displaystyle\sum_{i=1}^{n}\langle x,d\Phi_{p}(e_{i})\rangle P(d\Phi_{p}(e_{i}))\\
        &=\displaystyle\sum_{i=1}^{n}\langle P(x),P(d\Phi_{p}(e_{i}))\rangle P(d\Phi_{p}(e_{i}))\\
        &=\displaystyle\sum_{i=1}^{n}\langle P(x),d\Phi_{p}(P^{\Sigma}(e_{i}))\rangle d\Phi_{p}(P^{\Sigma}(e_{i}))\\
        &=[P(x)]^{\Sigma},
    \end{aligned}
    \end{equation*}where we have used the previous item that $P$ is an isometry and $\displaystyle\{P^{\Sigma}(e_{i})\}_{i=1}^{n}$ is still an orthonormal basis of $T_{p}\Sigma$.\\
    \end{enumerate}
    Let $X$, $Y$ be vector fields in $D_{1}$. Then

\begin{equation*}
\begin{aligned}
      \displaystyle d\Phi((\nabla_{X}P^{\Sigma})Y)
      &=d\Phi(\nabla_{X}P^{\Sigma}(Y)-P^{\Sigma}(\nabla_{X}Y))\\
      &=d\Phi(\nabla_{X}P^{\Sigma}(Y))-d\Phi(P^{\Sigma}(\nabla_{X}Y))\\
      \displaystyle &=[\bar{\nabla}_{d\Phi(X)}d\Phi(P^{\Sigma}(Y))]^{\Sigma}
      -P(d\Phi(\nabla_{X}Y))\\
      \displaystyle &=[\bar{\nabla}_{d\Phi(X)}P(d\Phi(Y))]^{\Sigma}
      -P([\bar{\nabla}_{d\Phi(X)}d\Phi(Y)]^{\Sigma})\\
       \displaystyle &=[\bar{\nabla}_{d\Phi(X)}P(d\Phi(Y))]^{\Sigma}
      -[P(\bar{\nabla}_{d\Phi(X)}d\Phi(Y))]^{\Sigma}\\
       \displaystyle &=[\bar{\nabla}_{d\Phi(X)}P(d\Phi(Y))
      -P(\bar{\nabla}_{d\Phi(X)}d\Phi(Y))]^{\Sigma}\\
      \displaystyle &=[(\bar{\nabla}_{d\Phi(X)}P)d\Phi(Y)
      ]^{\Sigma}=0.\\
      \end{aligned}
\end{equation*}
Since $d\Phi$ is one to one,  $(\nabla_{X}P^{\Sigma})Y=0$. Therefore, by Proposition 2 in \cite{gil1983geometric}, $D_{1}$ is a totally geodesic distribution.\\

Analogously, redefining $P^{\Sigma}:T_{p}\Sigma \rightarrow T_{p}\Sigma$ by $P^{\Sigma}(v_{1}+v_{2})=-v_{1}+v_{2}$, we conclude that $D_{2}$ is also a totally geodesic distribution. 
 
\end{itemize}
\end{proof}

  By Lemma \ref{lemmafact},  we get that   $\Sigma=\mathbb{CP}^{\frac{m_{1}}{2}}\times \psi^{-1}(s)$ up to an isometry, where $s\in \mathbb{CP}^{\frac{m_{1}}{2}}$ is a fixed element and that $\psi$ is the trivial submersion, i.e. $\psi(p)=r$, for $p=(r,q)$ where $r\in\mathbb{C}P^{\frac{m_{1}}{2}}$ and $q\in \psi^{-1}(s)$ (we pick an arbitrary element $s$ because  the fibers of $\psi$ are isometric, see \cite{hermann1960sufficient}).
  Notice that $\psi^{-1}(s)$ is a compact Riemannian manifold of dimension $m_{2}-1$ and that 
  
  \begin{center}
      $\displaystyle d\psi_{p}(x,w)=x$ for $x\in T_{r}\mathbb{CP}^{\frac{m_{1}}{2}}$ and $w\in T_{q}\psi^{-1}(s)$.
  \end{center}
  Now, we will show that the function $\phi$ does not depend on $r \in \mathbb{CP}^{\frac{m_{1}}{2}}$. It suffices to prove that $d\phi_{p}(x,0)=0$ for $x\in T_{r}\mathbb{CP}^{\frac{m_{1}}{2}}$. Let $x\in T_{r}\mathbb{C}P^{\frac{m_{1}}{2}}$. Then,
  \begin{center}
  $d\Phi_{p}(x,0)=(d\psi_{p}(x,0),d\phi_{p}(x,0))=(x,d\phi_{p}(x,0))$,
  \end{center}
which implies that
\begin{center}$|d\Phi_{p}(x,0)|^{2}=|x|^{2}+|d\phi_{p}(x,0)|^{2}$.
\end{center}
On the other hand, since $\Phi$ is an isometric immersion, 
 
 \begin{center}
     $|d\Phi_{p}(x,0)|^{2}=|(x,0)|^{2}=|x|^{2}.$
 \end{center}
 Thus, $|d\phi_{p}(x,0)|^{2}=0$. Since $\phi$ does not depend of $r \in \mathbb{CP}^{\frac{m_{1}}{2}}$, we can fix $r=s$ and denote
 
 \begin{center}$\hat{\phi}:\psi^{-1}(s)\rightarrow M$\\
 \hspace{3.3cm}$q \rightarrow \hat{\phi}(q):=\phi(s,q)$,
 \end{center}
 with $d\hat{\phi}_{q}(w)=d\phi_{p}(0,w)$ for $p=(s,q)$, $q\in \psi^{-1}(s)$. Now, it only remains to prove that  $\hat{\phi}$ is a stable minimal immersion  in $M$.
 \begin{itemize}
     \item If $d\hat{\phi}_{q}(w)=0$, for $w\in T_{q}\psi^{-1}(s)$, then 
 \begin{center}
     $d\Phi_{p}(0,w)=(d\psi_{p}(0,w),d\phi_{p}(0,w))=(0,0)$.
 \end{center}
 Since $\Phi$ is an immersion, $(0,w)=(0,0)$, and thus $w=0$.
 
 \item Let $w,v \in T_{q}\psi^{-1}(s)$,  
 \begin{equation*}
 \begin{aligned}
 \langle d\hat{\phi}_{q}(w),d\hat{\phi}_{q}(v) \rangle &=\langle d\phi_{p}(0,w),d\phi_{p}(0,v) \rangle\\
 &=\langle (0,d\phi_{p}(0,w)),(0,d\phi_{p}(0,v)) \rangle\\
 &=\langle (d\psi_{p}(0,w),d\phi_{p}(0,w)),(d\psi_{p}(0,v),d\phi_{p}(0,v)) \rangle\\
 &=\langle d\Phi_{p}(0,w),d\Phi_{p}(0,v)\rangle\\ &= \langle (0,w),(0,v)\rangle\\
 &=\langle w,v\rangle
 \end{aligned}
 \end{equation*}
 \item
 
 Let $H$ be the mean curvature vector of $\Phi$, $p=(r,q)\in \Sigma$, $N=(0,N^{2})$ be a local unitary normal vector field around $p$,  $X_{1},\ldots,X_{m_{1}}$ be local orthonormal vector fields around $r$ in $\mathbb{C}P^{\frac{m_{1}}{2}}$,  and  $W_{1},\ldots, W_{m_{2}-1}$ be local orthonormal vector fields around $q$ in $\psi^{-1}(s)$. We will omit the evaluation at $p$ in the following computation,

 \begin{equation*}
     \begin{aligned} 0=H &=\sum_{i=1}^{m_{1}}[\bar{\nabla}_{d\Phi(X_{i},0)}d\Phi(X_{i},0)]^{N\Sigma}+\sum_{j=1}^{m_{2}-1}[\bar{\nabla}_{d\Phi(0,W_{j})}d\Phi(0,W_{j})]^{N\Sigma}\\
     &=\sum_{i=1}^{m_{1}}[\bar{\nabla}_{(X_{i},0)}(X_{i},0)]^{N\Sigma}+\sum_{j=1}^{m_{2}-1}[\bar{\nabla}_{(0,d\hat{\phi}(W_{j}))}(0,d\hat{\phi}(W_{j}))]^{N\Sigma}\\
     & =\displaystyle\sum_{i=1}^{m_{1}}\langle \bar{\nabla}_{(X_{i},0)}(X_{i},0),(0,N^{2})\rangle(0,N^{2}) +\sum_{j=1}^{m_{2}-1}\langle \bar{\nabla}_{(0,d\hat{\phi}(W_{j}))}(0,d\hat{\phi}(W_{j})),(0,N^{2})\rangle(0,N^{2})\\
     &=\sum_{i=1}^{m_{1}}\langle(\nabla^{1}_{X_{i}}X_{i},0),(0,N^{2})\rangle(0,N^{2})
     +\sum_{j=1}^{m_{2}-1}\langle(0,\nabla^{2}_{d\hat{\phi}(W_{j})}d\hat{\phi}(W_{j})),(0,N^{2})\rangle(0,N^{2})\\
     & =\sum_{j=1}^{m_{2}-1}\langle(0,\nabla^{2}_{d\hat{\phi}(W_{j})}d\hat{\phi}(W_{j})),(0,N^{2})\rangle(0,N^{2})\\
     &=(0,\sum_{j=1}^{m_{2}-1}\langle \nabla^{2}_{d\hat{\phi}(W_{j})}d\hat{\phi}(W_{j}),N^{2}\rangle N^{2}),
     \end{aligned}
 \end{equation*}
 where $\nabla^{1}$ and $\nabla^{2}$ are the connections on $\mathbb{C}P^{\frac{m_{1}}{2}}$ and $M$, respectively. Then, 
 \begin{center}
 $\displaystyle A:=\sum_{j=1}^{m_{2}-1}\langle\nabla^{2}_{d\hat{\phi}(W_{j})}d\hat{\phi}(W_{j}),N^{2}\rangle N^{2}=0$.
\end{center} 
 Notice that $A$ is the mean curvature vector of the immersion $\hat{\phi}$, and thus $\hat{\phi}$ is minimal. Finally, $\hat{\phi}$ is  stable because $\Sigma$ is stable (see preliminaries section in \cite{torralbo2014stable}).
 \end{itemize}
\end{proof}

\begin{corollary}
There are no compact stable minimal hypersurfaces in the product manifold $\mathbb{C}P^{\frac{m_{1}}{2}}\times S^{s}$, $\mathbb{C}P^{\frac{m_{1}}{2}}\times \mathbb{O} P^{2}$, or $\mathbb{C}P^{\frac{m_{1}}{2}}\times \mathbb{K}P^{s}$, where $\mathbb{K}\in \{\mathbb{C}, \mathbb{H}\}$ other than  $\mathbb{C}P^{\frac{m_{1}}{2}}\times \{q\}$ in $\mathbb{C}P^{\frac{m_{1}}{2}}\times S^{1}$, where $q\in S^{1}$.
\end{corollary}
\begin{proof}
We apply Theorem \ref{cod1c}, and notice that the are no stable minimal hypersurfaces in $S^{s}$,  $\mathbb{C}P^{s}$, 
$\mathbb{H}P^{s}$, or $\mathbb{O} P^{2}$ (see Theorems \ref{t1.1} and \ref{ohnita}) other than a point $q$ in $S^{1}$.
\end{proof}

\begin{corollary}
The only compact stable minimal hypersurface in the product space $\mathbb{C}P^{\frac{m_{1}}{2}}\times \mathbb{R}P^{s}$ is $\mathbb{C}P^{\frac{m_{1}}{2}}\times \mathbb{R}P^{s-1}$.
\end{corollary}

\begin{proof}
We apply Theorem \ref{cod1c}, and notice that the only stable minimal hypersurface in $\mathbb{R}P^{s}$ is
$\mathbb{R}P^{s-1}$ (see Theorem \ref{ohnita0}).
\end{proof}

\subsection{CODIMENSION $2$ OR DIMENSION $2$}
\label{CC2}
In this subsection, we will use the general formula obtained in Subsection \ref{gfc} to prove a characterization of compact stable minimal immersions of codimension or dimension two in the product of two complex projective spaces.

\begin{lemma}
\label{cod2c}
Let $\Phi=(\psi,\phi):\Sigma \rightarrow \mathbb{C}P^{\frac{m_{1}}{2}}\times M$ be a compact stable minimal immersion of codimension $d=2$, where $M$ is any Riemannian manifold of dimension $m_{2}$. Then, $\eta_{2}^{1}=\pm J(\eta_{1}^{1})$, where $\{\eta_{1},\eta_{2}\}$ is an orthonormal basis of $N_{p}\Sigma$, $ p \in \Sigma$, $J$ is the complex structure of $\mathbb{C}P^{\frac{m_{1}}{2}}$ and $\eta_{i}^{1}$ is the projection of $\eta_{i}$ in $T_{\psi(p)}\mathbb{C}P^{\frac{m_{1}}{2}}$.
\end{lemma}

\begin{proof} Let $p\in \Sigma$. Since $d=2$, Equation (\ref{MAIN}) becomes

\begin{equation}
\begin{aligned}
\label{MAIND22}
\displaystyle\sum_{A=1}^{m} -\langle N_{E_{A}},J_{\Sigma}(N_{E_{A}})\rangle &=\lambda^{2}\Big ( \sum_{k=1}^{2}\sum_{l=1}^{2}\langle J(\eta^{1}_{k}),\eta^{1}_{l}\rangle^{2}-\langle \eta^{1}_{k},\eta^{1}_{l}\rangle^{2}\Big )\\
&=\lambda^{2}\Big (-|\eta_{1}^{1}|^{4}-|\eta_{2}^{1}|^{4}-2\langle\eta_{1}^{1},\eta_{2}^{1}\rangle^{2}+2\langle J(\eta_{1}^{1}),\eta_{2}^{1}\rangle^{2}\Big ).
\end{aligned}
\end{equation}

If $\eta_{1}^{1}=0$,

\begin{equation}
\label{TRp2}
   \displaystyle \sum_{A=1}^{m} -\langle N_{E_{A}},J_{\Sigma}(N_{E_{A}})\rangle=-\lambda^{2}|\eta_{2}^{1}|^{4}\leq 0.
\end{equation}

If $\eta_{1}^{1}\neq 0$, we can write $\eta_{2}^{1}$ in terms of $\eta_{1}^{1}$, $J(\eta_{1}^{1})$ and $X$ for some vector $X\in T_{\psi(p)}\mathbb{C}P^{\frac{m_{1}}{2}}$ which is unitary and orthogonal to  $\eta_{1}^{1}$ and $J(\eta_{1}^{1})$. In fact,

\begin{equation}
\label{eta12}
    \displaystyle \eta_{2}^{1}=
    \langle \eta_{2}^{1},\eta_{1}^{1}\rangle\frac{\eta_{1}^{1}}{|\eta_{1}^{1}|^{2}}
    +\langle\eta_{2}^{1},J(\eta_{1}^{1})\rangle\frac{J(\eta_{1}^{1})}{|\eta_{1}^{1}|^{2}}+\langle\eta_{2}^{1},X\rangle X.
\end{equation}
Hence,

\begin{center}
    $\displaystyle |\eta_{2}^{1}|^{2}= \frac{\langle\eta_{2}^{1},\eta_{1}^{1}\rangle^{2}}{|\eta_{1}^{1}|^{2}}
    +\frac{\langle\eta_{2}^{1},J(\eta_{1}^{1})\rangle^{2}}{|\eta_{1}^{1}|^{2}}+ \langle\eta_{2}^{1},X\rangle^{2}$,
\end{center}
and multiplying both sides by $|\eta_{1}^{1}|^{2}$,

\begin{center}
    $\displaystyle |\eta_{1}^{1}|^{2}|\eta_{2}^{1}|^{2}= \langle\eta_{2}^{1},\eta_{1}^{1}\rangle^{2}
    +\langle\eta_{2}^{1},J(\eta_{1}^{1})\rangle^{2}+ |\eta_{1}^{1}|^{2}\langle\eta_{2}^{1},X\rangle^{2}$.
\end{center}
Therefore,

\begin{center}
    $\displaystyle \langle\eta_{2}^{1},J(\eta_{1}^{1})\rangle^{2}= |\eta_{1}^{1}|^{2}|\eta_{2}^{1}|^{2}-\langle\eta_{2}^{1},\eta_{1}^{1}\rangle^{2}-|\eta_{1}^{1}|^{2}\langle\eta_{2}^{1},X\rangle^{2}.$
\end{center}
Replacing this last equation in Equation (\ref{MAIND22}), we have

\begin{equation}
\begin{aligned}
\label{TRp3}
\displaystyle\sum_{A=1}^{m} -\langle N_{E_{A}},J_{\Sigma}(N_{E_{A}})\rangle &=\lambda^{2}\Big (-|\eta_{1}^{1}|^{4}-|\eta_{2}^{1}|^{4}-2\langle\eta_{1}^{1},\eta_{2}^{1}\rangle^{2}
+2|\eta_{1}^{1}|^{2}|\eta_{2}^{1}|^{2}-2\langle\eta_{2}^{1},\eta_{1}^{1}\rangle^{2}-2|\eta_{1}^{1}|^{2}\langle\eta_{2}^{1},X\rangle^{2}\Big )\\ 
&=\lambda^{2}\Big (-|\eta_{1}^{1}|^{4}-|\eta_{2}^{1}|^{4}-4\langle\eta_{1}^{1},\eta_{2}^{1}\rangle^{2}
+2|\eta_{1}^{1}|^{2}|\eta_{2}^{1}|^{2}-2|\eta_{1}^{1}|^{2}\langle\eta_{2}^{1},X\rangle^{2}\Big )\\
\displaystyle &=\lambda^{2}\Big (-(|\eta_{1}^{1}|^{2}-|\eta_{2}^{1}|^{2})^{2}-4\langle\eta_{1}^{1},\eta_{2}^{1}\rangle^{2}
-2|\eta_{1}^{1}|^{2}\langle\eta_{2}^{1},X\rangle^{2}\Big )\leq 0.
\end{aligned}
\end{equation}

From Equations (\ref{TRp2}) and (\ref{TRp3}), we have
\begin{equation}
\label{signo}
    \displaystyle \sum_{A=1}^{m} -\langle N_{E_{A}},J_{\Sigma}(N_{E_{A}})\rangle\leq 0.
\end{equation}
Integrating Equation (\ref{signo}) and using the stability of $\Sigma$ gives us that
\begin{center}
$\displaystyle
\int_{\Sigma}\sum_{A=1}^{m} - \langle N_{E_{A}},J_{\Sigma}(N_{E_{A}})\rangle d\Sigma=0$.
\end{center}
From Equation (\ref{signo}), we know the integrand of the last equality has a sign, and hence

\begin{equation}
\label{F}
\sum_{A=1}^{m} - \langle N_{E_{A}},J_{\Sigma}(N_{E_{A}})\rangle=0.
\end{equation}

We have two options: if a point is such that $\eta_{1}^{1}=0$, using Equation (\ref{F}) in Equation (\ref{TRp2}), we have that $\eta_{2}^{1}=0$. On the other hand, if a point is such that $\eta_{1}^{1}\neq 0$, using Equation (\ref{F}) in Equation (\ref{TRp3}), we have that:

\begin{itemize}
    \item $|\eta_{1}^{1}|=|\eta_{2}^{1}|$
    \item $\langle \eta_{2}^{1},\eta_{1}^{1}\rangle=0$
    \item $\langle \eta_{2}^{1},X\rangle=0.$
\end{itemize}
Using the last two items in Equation (\ref{eta12}), we have

\begin{equation*}
    \displaystyle \eta_{2}^{1}=
    \langle\eta_{2}^{1},J(\eta_{1}^{1})\rangle\frac{J(\eta_{1}^{1})}{|\eta_{1}^{1}|^{2}}.
\end{equation*}
Therefore, $\eta_{2}^{1}$ and $J(\eta_{1}^{1})$ are parallel, and since $|J(\eta_{1}^{1})|=|\eta_{1}^{1}|=|\eta_{2}^{1}|$,

\begin{equation*}
    \eta_{2}^{1}=\pm J(\eta_{1}^{1}).
\end{equation*}
Notice we can include the case  $\eta_{1}^{1}=\eta_{2}^{1}=0$ in then last equality. 
\end{proof}

\begin{definition}
\label{structures}
Let $J_{1}$ and $J_{2}$ be two almost complex structures on $\mathbb{C}P^{\frac{m_{1}}{2}}\times \mathbb{C}P^{\frac{m_{2}}{2}} $ given by: 
\begin{center}
    $J_{1}(X,Y):=(J(X),J(Y))$ and $J_{2}(X,Y):=(J(X),-J(Y))$.
\end{center}
\end{definition}
\begin{definition}
Let $\Phi:\Sigma \rightarrow \mathbb{C}P^{\frac{m_{1}}{2}}\times \mathbb{C}P^{\frac{m_{2}}{2}}$ be an immersion and $p\in \Sigma$. For $i\in\{1,2\}$ fixed, we say that $T_{p}\Sigma\equiv d\Phi_{p}(T_{p}\Sigma)$ has structure $J_{i}$ if $J_{i}(T_{p}\Sigma)=T_{p}\Sigma$ or equivalently $J_{i}(N_{p}\Sigma)=N_{p}\Sigma$. If for all $p\in \Sigma$, we have that $T_{p}\Sigma$ has structure $J_{i}$, we say that $\Phi$ is a complex immersion under the structure $J_{i}$.
\end{definition}

\begin{theorem}
\label{z2}
Let $\Phi:\Sigma \rightarrow \bar{M}:= \mathbb{C}P^{\frac{m_{1}}{2}}\times \mathbb{C}P^{\frac{m_{2}}{2}} $ be a compact stable minimal immersion of codimension $d=2$ and dimension $n$. Then, $\Phi$ is a complex immersion under the structure $J_{1}$ or $J_{2}$.
\end{theorem}

\begin{proof}
From Lemma \ref{cod2c}, we have that for $q\in \Sigma$ and $\{\eta_{1},\eta_{2}\}$ an orthonormal basis of $N_{q}\Sigma$,

\begin{center}
    $\eta_{2}^{1}=\pm J(\eta_{1}^{1})$ and $\eta_{2}^{2}=\pm J(\eta_{1}^{2})$,
\end{center}
and then
\begin{center}
    $|\eta^{1}_{1}|=|\eta^{1}_{2}|$ and $|\eta^{2}_{1}|=|\eta^{2}_{2}|$.
\end{center}
\begin{remark}
\label{2s}
Notice that if $q\in \Sigma$ is such that $|\eta_{1}^{1}|=|\eta_{2}^{1}|=0$ or $|\eta_{1}^{2}|=|\eta_{2}^{2}|=0$,   $T_{q}\Sigma$ has both structures $J_{1}$ and $J_{2}$. 
\end{remark}
If for all $q \in \Sigma$, we have that $|\eta_{1}^{1}|=|\eta_{2}^{1}|=0$ or $|\eta_{1}^{2}|=|\eta_{2}^{2}|=0$, then all points of $\Sigma$ have both structures and then we are done. Otherwise there exists a point $p\in \Sigma$ such that 
\begin{equation}
\label{sig}
|\eta_{1}^{1}|=|\eta_{2}^{1}|\neq0 \hspace{0.2cm} \text{and} \hspace{0.2cm} |\eta_{1}^{2}|=|\eta_{2}^{2}|\neq0,
\end{equation}
for $\{\eta_{1},\eta_{2}\}$ an orthonormal basis for $N_{p}\Sigma$. Denote by $cut(p)$ the cut locus of the point $p$ in $\Sigma$. \\

In Part I we will show that $\Sigma\setminus cut(p)$ has a single complex structure and in Part II we will show that we can extend this complex structure to the set $cut(p)$.\\

\textbf{Part I}\\
By parallel transport of the  orthonormal basis $\{\eta_{1},\eta_{2}\}$ in $N_{p}\Sigma$ along geodesics of $\Sigma$ under the normal connection of $\Sigma$ in $\bar{M}$, we define normal vector fields $N_{1}$ and $N_{2}$ in $\Sigma\setminus cut(p)$. Using Lemma \ref{cod2c}, we have that for an arbitrary point $\tau \in \Sigma\setminus cut(p)$,\\

\textbf{Set of equations I}
\begin{equation}
    \label{p1}
    N_{2}^{1}(\tau)=J(N_{1}^{1})(\tau)
\end{equation}
or 
\begin{equation}
    \label{p2}
    N_{2}^{1}(\tau)=-J(N_{1}^{1})(\tau)
\end{equation}
and 
\begin{equation}
    \label{q1}
    N_{2}^{2}(\tau)=J(N_{1}^{2})(\tau)
\end{equation}
or 
\begin{equation}
    \label{q2}
    N_{2}^{2}(\tau)=-J(N_{1}^{2})(\tau).
\end{equation}
Hence, there are four possible options at $\tau \in \Sigma \setminus cut(p)$:\\

\textbf{Set of equations II}

\begin{itemize}
    \item If $\tau$ satisfies Equations (\ref{p1}) and (\ref{q2}), we have
    \begin{equation}
    \label{r1}
        J_{2}(N_{1})(\tau)=N_{2}(\tau).
    \end{equation}
    
    \item  If $\tau$ satisfies Equations (\ref{p2}) and (\ref{q2}), we have
    \begin{equation}
    \label{r2}
        J_{1}(N_{1})(\tau)=-N_{2}(\tau).
    \end{equation}

    \item  If $\tau$ satisfies Equations (\ref{p1}) and (\ref{q1}), we have
    \begin{equation}
    \label{r3}
        J_{1}(N_{1})(\tau)=N_{2}(\tau).
    \end{equation}
    
    \item  If $\tau$ satisfies Equations (\ref{p2}) and (\ref{q1}), we have
    \begin{equation}
    \label{r4}
        J_{2}(N_{1})(\tau)=-N_{2}(\tau).
    \end{equation}
    
\end{itemize}

Therefore, for $\tau\in \Sigma\setminus cut(p)$, $T_{\tau}\Sigma$ has the structure $J_{1}$ or $J_{2}$.
Without loss of generality, assume that $p$ has structure $J_{1}$. We will prove that all the points in  $\Sigma\setminus cut(p)$ have the same structure $J_{1}$. 

If all the points in $\Sigma\setminus cut(p)$ are such that their tangent space has the structure $J_{1}$, then we are done.  Otherwise there exists a point $r$ in $\Sigma \setminus cut(p)$ such that $T_{r}\Sigma$ has complex structure $J_{2}$ and not $J_{1}$. Recall that from Equation (\ref{sig}) and the construction of $N_{1}$ and $N_{2}$ we have 
\begin{equation}
\label{z0}
    |N_{1}^{1}|(p)=|N_{2}^{1}|(p)\neq0 \hspace{0.2cm}\text{and}\hspace{0.2cm} |N_{1}^{2}|(p)=|N_{2}^{2}|(p)\neq0. 
\end{equation}
And since $T_{r}\Sigma$ does not have the structure $J_{1}$, by Remark \ref{2s}, 
\begin{equation}
\label{z1}
    |N_{1}^{1}|(r)=|N_{2}^{1}|(r)\neq0 \hspace{0.2cm}\text{and}\hspace{0.2cm} |N_{1}^{2}|(r)=|N_{2}^{2}|(r)\neq0. 
\end{equation}
Without loss of generality, suppose $p$ satisfies Equation (\ref{r2}), i.e. $J_{1}(N_{1})(p)=-N_{2}(p)$ and $r$ satisfies Equation (\ref{r1}), i.e. $J_{2}(N_{1})(r)=N_{2}(r)$ (see Remark \ref{tabla}). According to the set of equations I, 

\begin{equation}
\label{w0}
J(N_{1}^{1})(p)=-N^{1}_{2}(p)\text{,}\hspace{0.2cm} J(N_{1}^{2})(p)=-N_{2}^{2}(p), 
\end{equation}
    and
\begin{equation}
\label{w1}
    J(N_{1}^{1})(r)=N_{2}^{1}(r)\text{,}\hspace{0.2cm} J(N_{1}^{2})(r)=-N_{2}^{2}(r).
\end{equation}
Let $\gamma:=\gamma(s,p,v):[0,1]\rightarrow \Sigma$ be the unique geodesic contained in $\Sigma\setminus cut(p)$ such that $\gamma(0)=p$ and $\gamma(1)=r$. Let us define $f:[0,1]\rightarrow \mathbb{R}$ in the following way,

\begin{center}
    $f(s):=\langle J(N^{1}_{1}),N_{2}^{1}\rangle_{\gamma(s)}$.
\end{center}
Using Equations (\ref{w0}) and (\ref{z0}) and Equations (\ref{w1}) and (\ref{z1}), $f$ is such that
\begin{center}
    $f(0)=\langle J(N^{1}_{1}),N_{2}^{1}\rangle_{p}=-|N_{2}^{1}|^{2}(p)< 0$, and\\ $f(1)=\langle J(N^{1}_{1}),N_{2}^{1}\rangle_{r}=|N_{2}^{1}|^{2}(r)>0$,
\end{center}
respectively. Since $f$ is smooth, there exists a point $t\in (0,1)$ such that,
\begin{center}
$f(t)=0$, $f(s)>0$ for $s\in (t,t+\epsilon)$ and $f(s)\leq 0$ for $s\in (t-\epsilon,t)$. 
\end{center}
Let $g:(t-\epsilon,t+\epsilon)\rightarrow \mathbb{R}$ be the function given by $g(s):= |N_{2}^{1}|_{\gamma(s)}^{2}$. Since $J(N_{1}^{1})=\pm N_{2}^{1}$,

\begin{center}
$
f(s)= \left\{ \begin{array}{lcc}
             -g(s)&   if  & s\in (t-\epsilon,t] \\
             \\ g(s) &  if & s\in (t,t+\epsilon).
             \end{array}
   \right.$
\end{center}
Since $f$ is smooth,  for all $k\geq 0, k\in \mathbb{N}$
\begin{center}
$
f^{(k)}(s)= \left\{ \begin{array}{lcc}
             -g^{(k)}(s)&   if  & s\in (t-\epsilon,t] \\
             \\ g^{(k)}(s) &  if & t\in (t,t+\epsilon)
             \end{array}
   \right.$
\end{center}
is continuous, and thus $g^{(k)}(t)=0$.\\

The function $g$ is real analytic because it is a composition of real analytic functions (see the technical Lemma \ref{realanal}). Since $g$ is an analytic function such that $g^{(k)}(t)=0$ for all $k\geq 0$, $g=0$. This is a contradiction because $g(t+\frac{\epsilon}{2})=f(t+\frac{\epsilon}{2})>0$. Therefore, all the points in $\Sigma\setminus cut(p)$ are such that their tangent space has the structure $J_{1}$.\\

The proof of Lemma \ref{realanal} can be omitted on a first reading; the reader may wish to continue to Part II.

\begin{lemma} \label{realanal} The function $g:(t-\epsilon,t+\epsilon)\rightarrow \mathbb{R}$ given by  $g(s):=|N_{2}^{1}|_{\gamma(s)}^{2}$ is a real analytic function.
\begin{proof}
Let 
\begin{equation}
\label{chart}
\begin{aligned}
    \displaystyle Y:U\subset \mathbb{R}^{n+d} & \rightarrow \bar{U}\subset \bar{M}\\
    (y^{1},\ldots,y^{n+d}) &\rightarrow Y(y^{1},\ldots,y^{n+d})
    \end{aligned}
\end{equation}
 be a coordinate chart of $\bar{M}$ around $\Phi(\gamma(t))$ compatible with the real analytic structure of $\bar{M}$ and 
 \begin{equation*}
 \begin{aligned}
     X:W\subset \mathbb{R}^{n} &\rightarrow \bar{W}\subset \Sigma \\
     (x^{1},\ldots,x^{n})&\rightarrow X(x^{1},\ldots,x^{n})
     \end{aligned}
 \end{equation*} 
 be the coordinate chart of $\Sigma$ around $\gamma(t)$. Notice that we can assume that $\gamma (t-\epsilon,t+\epsilon)\subset \bar{W}$, otherwise we can just modify $\epsilon$.  
If $y$ is the local representation of $\Phi$ in the coordinate charts $X$ and $Y$,
\begin{equation*}
\begin{aligned}
y:W &\rightarrow \mathbb{R}^{n+d}\\
x=(x^{1},\ldots,x^{n})&\rightarrow y(x)=(y^{1}(x),\ldots,y^{n+d}(x)).
\end{aligned}
\end{equation*}
We also can assume that $y(W)\subset U$, otherwise we can make the open set $W$ smaller. \\

Since $\Phi$ is a minimal immersion, $y$ satisfies the following non-linear elliptic system of partial differential equations (see Section 52 in \cite{eisenhart2016riemannian}):

\begin{equation}
\label{minimal}
     g^{ij}\frac{\partial^{2} y^{\beta}}{\partial x^{i}\partial x^{j} }+\frac{g^{ij}a^{\beta\sigma}}{2}(\frac{\partial a_{\mu \sigma}}{\partial y^{v}}+\frac{\partial a_{v \sigma}}{\partial y^{\mu}}-\frac{\partial a_{\mu v}}{\partial y^{\sigma}})\frac{\partial y^{\mu}}{\partial x^{i}}\frac{\partial y^{v}}{\partial x^{j}}=0,\hspace{0.5cm }\beta=1,\ldots,n+d,
\end{equation}
where $i,j$ are summing in $\{1,\ldots,n\}$, $\sigma,\mu,v$ are summing in $\{1,\ldots,n+d\}$, $a_{\alpha \beta}$ is the Riemannian metric in $\bar{M}$, and $g_{ij}$ is the metric in $\Sigma$ which is given by:
\begin{equation}
\label{metric}
\displaystyle g_{ij}=a_{\alpha\theta}\frac{\partial y^{\alpha}}{\partial x^{i}} \frac{\partial y^{\theta}}{\partial x^{j}}. 
\end{equation}
The system (\ref{minimal}) is real analytic because $\bar{M}$ is a real analytic Riemannian manifold. Then, using the main result in \cite{morrey1958analyticity}, the local representation of $\Phi$, $y(x)$, is also real analytic. We can also conclude, from (\ref{metric}) that the metric $g_{ij}$ is real analytic as a function of $x$. In the same manner we can see  that the local representation $\bar{\gamma}:=X^{-1}\circ \gamma$ of the geodesic $\gamma$ in $\Sigma$ around $\gamma(t)$ is real analytic. \\

Using the local coordinates described above we can see that $g(s)$ is the composition of the following functions:
\begin{center}

    $s\xrightarrow{\bar{\gamma}} \bar{\gamma}(s) \xrightarrow{N_{2}} N_{2}(\bar{\gamma}(s))\xrightarrow{P^{1}} N_{2}^{1}(\bar{\gamma}(s))\xrightarrow{|.|^{2}} |N_{2}^{1}(\bar{\gamma}(s))|^{2}$.
\end{center}
We already proved that  $\bar{\gamma}$ is a real analytic function. So we now we have to prove that the other functions described above are real analytic too.\\

For convenience, we will take the coordinate chart (\ref{chart}) given by slice coordinates for $\Sigma$ in $\bar{M}$ around $\Phi(\gamma(t))$. By using the process of Gram–Schmidt we can construct in $U$ vector fields $E_{1},E_{2}$ that are real analytic as functions of $(y^{1},\ldots,y^{n+d})\in U$ and such that for every point $x\in W$, $\{E_{1}(y(x)),E_{2}(y(x))\}$ is an orthonormal basis of $N_{x}\Sigma$. Therefore,  $N_{2}$ can be seen as,
    \begin{equation}
    N(s):=N_{2}(\bar{\gamma}(s))=c^{z}(s)E_{z}(y(\bar{\gamma}(s))); \hspace{0.2cm}z=1,2.    
    \end{equation}
Recall $N_{2}$ is the parallel transport of a fixed unitary vector $\eta_{2}\in N_{p}\Sigma$ along the geodesic $\gamma(s)=\gamma(s,p,v)$ under the normal connection. Then, 

\begin{equation}
\label{tp}
    \nabla^{\perp}_{\frac{d \bar{\gamma}}{ds}}N(s)=0, \hspace{0.5cm} N(s_{0})=\eta,
\end{equation}
for some $s_{0}\in (t-\epsilon,t+\epsilon)$ and $\eta \in N_{\bar{\gamma}(s_{o})}\Sigma$. If $\eta=a^{z}E_{z}(y(\bar{\gamma}(s_{0})))$, for $a^{z}\in \mathbb{R}$, from (\ref{tp}) we have the following system of ordinary differential equations,\\

$\displaystyle
\left\{ \begin{array}{lcc}
             \displaystyle  \frac{d c^{z}(s)}{d s }=-c^{l}(s)\langle\bar{\nabla}_{\frac{d \bar{\gamma}}{ds}}E_{l}(y(\bar{\gamma}(s))),E_{z}(y(\bar{\gamma}(s)))\rangle:=G_{z}(c,s) \\
             \\c^{z}(s_{0})=a^{z},\hspace{0.5cm} z=1,2
             \end{array}
            \right.$\\
where $c=(c^{1},c^{2})$. Since $G:=(G_{1},G_{2})$ is  real analytic as a function of $(c,s)$, the solution $c(s)$ is real analytic, and therefore $N(s)=N_{2}(\bar{\gamma}(s))$ is real analytic. \\

It is straightforward to see that the projection $P^{1}$ of a real analytic vector field is also real analytic. Therefore, 
\begin{center}
$\displaystyle N^{1}_{2}(\bar{\gamma}(s))=\sum_{\beta=1}^{n+d}b^{\beta}(y(\bar{\gamma}(s)))\frac{\partial}{\partial y^{\beta}}\biggr\vert_{y(\bar{\gamma}(s))}$,
\end{center}
where $b^{\beta}$ are real analytic functions. Now, 

\begin{center}
    $\displaystyle|N^{1}_{2}(\bar{\gamma}(s))|^{2}=\sum_{\alpha,\beta=1}^{n+d}b^{\alpha}(y(\bar{\gamma}(s)))b^{\beta}(y(\bar{\gamma}(s)))a_{\alpha \beta}(y(\bar{\gamma}(s)))$,
\end{center}
which involves only products, sums, and compositions of real analytic functions. Therefore,  $g(s)$ is real analytic.

\end{proof}
\end{lemma}

\textbf{Part II}\\
Now we will show that we can extend that structure to the points in $cut(p)$.
\begin{remark}
\label{localstructure}
Notice that in Part I we have shown that if a point $q\in \Sigma$ is such that $T_{q}\Sigma$ has an structure ($J_{1}$ or $J_{2}$), and if for $\{\eta_{1},\eta_{2}\}$ a basis of $N_{q}\Sigma$ we have that 
\begin{equation}
\label{condition}
|\eta_{1}^{1}|=|\eta_{2}^{1}|\neq0 \hspace{0.2cm} \text{and} \hspace{0.2cm} |\eta_{1}^{2}|=|\eta_{2}^{2}|\neq0,
\end{equation}
then $\Sigma \setminus cut(q)$ has the same structure as $q$.
\end{remark}

Let $b$ be a point in  $cut(p)$. If $b$ is such that one of the projections of its normal vectors is zero, by Remark \ref{2s}, $b$ has both structures and we are done. Otherwise, $b$ is such that  (\ref{condition}) is satisfied for $\{\eta_{1},\eta_{2}\}$ a basis of $N_{b}\Sigma$. Let $V$ be a normal neighborhood around $b$ in $\Sigma$, where we can define orthonormal normal vector fields $N_{1}$ and $N_{2}$ such that 
\begin{equation}
\label{secondcondition}
|N_{1}^{1}|=|N_{2}^{1}|\neq0 \hspace{0.2cm} \text{and} \hspace{0.2cm} |N_{1}^{2}|=|N_{2}^{2}|\neq0,
\end{equation}
with $N_{1}(b)=\eta_{1}$ and $N_{2}(b)=\eta_{2}$. Let $\alpha:[0,l]\rightarrow \Sigma$ be a geodesic of $\Sigma$, such that $\alpha(0)=p$, $\alpha(l)=b$ and $\alpha([0,l])\cap cut(p)=b$.
 There exists $a<l$ such that $\alpha(a)\in V$.
Since $\alpha(a)\notin cut(p)$,  $\alpha(a)$ has structure $J_{1}$. Moreover, since $\alpha(a)\in V$, $b\notin cut(\alpha(a))$, and therefore by Remark \ref{localstructure} and (\ref{secondcondition}), $b$ has the same structure as $\alpha(a)$, i.e.,  $J_{1}$.

\end{proof}

\begin{remark}
\label{tabla}For the other cases, we have the following table,
\begin{center}
 \begin{tabular}{|c|c|c|c|} 
 \hline
 \hspace{0.2cm} Equation satisfied by $p$ \hspace{0.2cm}& \hspace{0.2cm} Equation satisfied by $r$ \hspace{0.2cm} & $f(s)$ & $g(s)$ \\
 \hline
 (\ref{r2})& (\ref{r4}) & $\langle J(N_{1}^{2}),N_{2}^{2}\rangle_{\gamma(s)}$ & $|N_{2}^{2}|_{\gamma(s)}^{2}$ \\ 
 \hline
 (\ref{r3}) & (\ref{r1}) & $-\langle J(N_{1}^{2}),N_{2}^{2}\rangle_{\gamma(s)}$ & $|N_{2}^{2}|_{\gamma(s)}^{2}$ \\
 \hline
 (\ref{r3}) & (\ref{r4}) & $-\langle J(N_{1}^{1}),N_{2}^{1}\rangle_{\gamma(s)}$ & $|N_{2}^{1}|_{\gamma(s)}^{2}$ \\ [1ex] 
 \hline
\end{tabular}
\end{center}
\end{remark}

Using the same arguments used in the proofs of Lemma \ref{cod2c} and  Theorem \ref{z2} and using Equation (\ref{MAIN2}), we have the following:

\begin{lemma}
Let $\Phi=(\psi,\phi):\Sigma \rightarrow \mathbb{C}P^{\frac{m_{1}}{2}}\times M$ be a compact stable minimal immersion of dimension $n=2$, where $M$ is any Riemannian manifold of dimension $m_{2}$. Then, $e_{2}^{1}=\pm J(e_{1}^{1})$, where $\{e_{1},e_{2}\}$ is an orthonormal basis of $T_{p}\Sigma$, $ p \in \Sigma$, $J$ is the complex structure of $\mathbb{C}P^{\frac{m_{1}}{2}}$, and $e_{i}^{1}$ is the projection of $e_{i}$ in $T_{\psi(p)}\mathbb{C}P^{\frac{m_{1}}{2}}$.
\end{lemma}

\begin{theorem}
Let $\Phi:\Sigma \rightarrow \mathbb{C}P^{\frac{m_{1}}{2}}\times \mathbb{C}P^{\frac{m_{2}}{2}}$ be a compact stable minimal immersion of dimension $n=2$. Then, $\Phi$ is a complex immersion under the structure $J_{1}$ or $J_{2}$.
\end{theorem}

\subsection{DIMENSION $1$}
\label{CD2}
In this subsection, we will use the general formula proved in Subsection \ref{gfc} to prove a classification theorem for compact stable minimal immersions of dimension $1$ (geodesics) in the product of a complex projective space with any other Riemannian manifold. Moreover, as an application, we obtain some corollaries when the second manifold is a compact rank one space.

\begin{theorem}
\label{dim1c}
Let $\Phi=(\psi,\phi):\Sigma\rightarrow \mathbb{C}P^{\frac{m_{1}}{2}}\times M$ be a compact stable minimal immersion of dimension $n=1$, where $M$ is any Riemannian manifold of dimension $m_{2}$. Then, $\phi:\Sigma \rightarrow M$ is a  stable geodesic, $\psi$ is a constant function, and therefore $\Phi(\Sigma)=\{r\}\times \phi(\Sigma)$ with $r$ a point of $\mathbb{C}P^{\frac{m_{1}}{2}}$.
\end{theorem}
\begin{proof}
Since $n=1$, Equation (\ref{MAIN2}) becomes

\begin{equation}
\label{G0}
\sum_{A=1}^{m} -\langle N_{E_{A}},J_{\Sigma}(N_{E_{A}})\rangle=-\lambda^{2} |e^{1}|^{4},
\end{equation}
where $e=(e^{1},e^{2})$ is an unitary vector in $d\Phi_{p}(T_{p}\Sigma)$, for $p\in\Sigma$. Therefore,
\begin{center}
$\displaystyle 0\leq  \sum_{A=1}^{m} -\int_{\Sigma}\langle N_{E_{A}},J_{\Sigma}(N_{E_{A}})\rangle d\Sigma=-\lambda^{2}\int_{\Sigma} |e^{1}|^{4}d\Sigma\leq 0,$
\end{center}
where we have used the fact that $\Sigma$ is stable in the first inequality. Hence, for $p\in\Sigma$, $e^{1}=0$. Therefore, 
\begin{equation}
\label{tsg}
    d\Phi_{p}(T_{p}\Sigma)=\{ \alpha (0,e^{2}):\alpha\in \mathbb{R}\}.
\end{equation}
Let $x\in T_{p}\Sigma$ arbitrary. Then, for some $\alpha\in \mathbb{R}$,

\begin{equation*}
    d\Phi_{p}(x)=(d\psi_{p}(x),d\phi_{p}(x))=(0,\alpha e^{2}).
\end{equation*}
Therefore, $d\psi_{p}(x)=0$, which implies that $\psi$ is constant. Now we will prove that $\phi:\Sigma \rightarrow M$ is a stable minimal immersion of dimension $1$. Since $\Phi$ is an isometric immersion and $\psi$ is constant, $\phi$ is an isometric immersion. From Equation (\ref{tsg}), we have that at $p\in \Sigma$,
\begin{center}
    $N_{p}\Sigma = \{ (v,0): v\in T_{\psi(p)}\mathbb{C}P^{\frac{m_{1}}{2}}\}\bigoplus \{(0,w): w\in [e^{2}]^{\perp_{M}}\}$.
\end{center}
Let $H$ be the mean curvature vector of $\Phi$, $E=(0,E^{2})$ a local unitary vector field tangent to $\Sigma$ around $p$ with $E(p)=(0,e^{2})$, and take the orthonormal basis of $N_{p}\Sigma$ given by 
\begin{center}
    $\{(v_{i},0):i=1,\ldots,m_{1}\}\bigcup \{(0,w_{j}): j=1,\ldots, m_{2}-1\}$,
\end{center}
 where $v_{i}\in T_{\psi(p)}\mathbb{C}P^{\frac{m_{1}}{2}}$ and $w_{j}\in [e^{2}]^{\perp_{M}}$. We will omit the evaluation at $p$ in the following computation, 

\begin{equation*}
\begin{aligned}
    \displaystyle 0=H &=(\bar{\nabla}_{E}E)^{N\Sigma}=(0,\nabla^{2}_{E^{2}}E^{2})^{N\Sigma}\\
    \displaystyle &=\sum_{i=1}^{m_{1}}\langle (0,\nabla^{2}_{E^{2}}E^{2}),(v_{i},0)\rangle(v_{i},0)+\sum_{j=1}^{m_{2}-1}\langle(0,\nabla^{2}_{E^{2}}E^{2}),(0,w_{j})\rangle(0,w_{j})\\
    \displaystyle &=\sum_{j=1}^{m_{2}-1}\langle(0,\nabla^{2}_{E^{2}}E^{2}),(0,w_{j})\rangle(0,w_{j})=
    (0,\sum_{j=1}^{m_{2}-1}\langle\nabla^{2}_{E^{2}}E^{2},w_{j}\rangle w_{j}).
    \end{aligned}
\end{equation*}
Therefore, 
\begin{center}
    $\displaystyle A:=\sum_{j=1}^{m_{2}-1}\langle\nabla^{2}_{E^{2}}E^{2},w_{j} \rangle w_{j}=0$.
\end{center}
But notice that $A$ is the mean curvature vector of $\phi$ as an immersion in $M$, and thus $\phi$ is minimal. Moreover, since $\Phi$ is stable then $\phi$ is stable (see preliminaries in Torralbo and Urbano \cite{torralbo2014stable}).

\end{proof}

\color{black}
\begin{corollary}
There are no compact stable geodesics in the product space $\mathbb{C}P^{\frac{m_{1}}{2}}\times S^{s}$, $\mathbb{C}P^{\frac{m_{1}}{2}}\times \mathbb{O} P^{2}$, or  $\mathbb{C}P^{\frac{m_{1}}{2}}\times \mathbb{K}P^{s}$, where $\mathbb{K}\in \{\mathbb{C}, \mathbb{H}\}$ other than $\{r\}\times S^{1}$ in $\mathbb{C}P^{\frac{m_{1}}{2}}\times S^{1}$, where $r\in \mathbb{C}P^{\frac{m_{1}}{2}}$. 
\end{corollary}
\begin{proof}
We apply Theorem \ref{dim1c} and notice that there are no stable geodesics in $S^{s}$, $\mathbb{C}P^{s}$, 
$\mathbb{H}P^{s}$, or $\mathbb{O} P^{2}$ (see Theorems \ref{t1.1} and \ref{ohnita}) other than $S^{1}$ in $S^{1}$.
\end{proof}

\begin{corollary}
The only compact stable geodesic in the product space $\mathbb{C}P^{\frac{m_{1}}{2}}\times \mathbb{R}P^{s}$ is $\{r\}\times \mathbb{R}P^{1}$, $r\in \mathbb{C}P^{\frac{m_{1}}{2}}$.
\end{corollary}

\begin{proof}
We apply Theorem \ref{dim1c} and notice that the only stable geodesic in $\mathbb{R}P^{s}$ is
$\mathbb{R}P^{1}$ (see Theorem \ref{ohnita0}).
\end{proof}

\section{MINIMAL STABLE SUBMANIFOLDS IN $\mathbb{H}P^{\frac{m_{1}}{4}}\times M$}

The strategies that will be used in this section are similar to those found in Section \ref{complexcompu}. Some details are presented again for the sake of completeness. 

\begin{subsection}{GENERAL FORMULA}
\label{quaternioniccompu}
Let $\Phi=(\psi,\phi):\Sigma\rightarrow \bar{M}:= \mathbb{H}P^{\frac{m_{1}}{4}}\times M$ be a  compact minimal immersion of codimension $d$ and dimension $n$, where $M$ is any Riemannian manifold of dimension $m_{2}$ and $\Phi_{1}:\mathbb{H}P^{\frac{m_{1}}{4}}\rightarrow \mathbb{R}^{m}$ is the immersion described in Section \ref{section1}. For each $v\in \mathbb{R}^{m}$ let us consider the following:
\begin{center}
    $\nu:=(v,0)\in T(\mathbb{R}^{m}\times M)$\\
    $N_{v}:=[\nu]^{N}$,
\end{center}
where $[.]^{N}$ is  projection in the orthogonal complement, $N_{p}\Sigma$, of $T_{p}\Sigma$ in $T_{\Phi(p)}\bar{M}$, $p\in \Sigma$.

\begin{lemma}
\label{GENERALH}
Let $p\in \Sigma$, $\{e_{1},\ldots,e_{n}\}$
be an orthonormal basis of  $T_{p}\Sigma$, $\{\eta_{1},\ldots,\eta_{d}\}$ be an orthonormal basis of  $N_{p}\Sigma$, and $\{E_{1},\ldots,E_{m}\}$ be the usual canonical basis of $\mathbb{R}^{m}$. Then, for $s\in \{1,2,3\}$
\begin{center}
    $\displaystyle\sum_{A=1}^{m} -\langle N_{E_{A}},J_{\Sigma}(N_{E_{A}})\rangle$
\end{center}
\begin{equation}
\label{hMAIN}=\lambda^{2}(
-\sum_{\substack{k=1\\
         k\neq s}}^{3}\sum_{j=1}^{n}\sum_{\beta=1}^{d} \langle J_{k}(\eta^{1}_{\beta}),e^{1}_{j} \rangle^{2}\displaystyle +\sum_{\beta=1}^{d}\sum_{l=1}^{d} \langle J_{s}(\eta^{1}_{\beta}),\eta^{1}_{l} \rangle^{2}-\langle \eta^{1}_{\beta},\eta^{1}_{l} \rangle^{2})
\end{equation}

\begin{equation}
    \label{hMAIN2}=\lambda^{2}(-\sum_{\substack{k=1\\
         k\neq s}}^{3}\sum_{j=1}^{n}\sum_{\beta=1}^{d} \langle J_{k}(\eta^{1}_{\beta}),e^{1}_{j}\rangle^{2}
+\sum_{i=1}^{n}\sum_{j=1}^{n}\langle J_{s}(e^{1}_{i}),e_{j}^{1}\rangle^{2}-\langle e^{1}_{i},e_{j}^{1}\rangle^{2}).
\end{equation}
\end{lemma}

\begin{proof}
Recall that $R$ is the curvature tensor of $\mathbb{H}P^{\frac{m_{1}}{4}}$ and $B$ is the second fundamental form of $\mathbb{H}P^{\frac{m_{1}}{4}}$ in $\mathbb{R}^{m}$ (see Subsection 1.2).
Using Equation (2.8) in \cite{chen2013stable} and proceeding as in the beginning of the proof of Lemma \ref{GENERALC} (changing the index $k$ for $\beta$)  we have the following equation (same equation than Equation (\ref{baseequation}))

\begin{equation}
\label{hbaseequation}
    \displaystyle\displaystyle\sum_{A=1}^{m} -\langle N_{E_{A}},J_{\Sigma}(N_{E_{A}})\rangle=\sum_{j=1}^{n}\sum_{\beta=1}^{d}-\frac{4}{3} \langle R(e^{1}_{j},\eta^{1}_{\beta})\eta^{1}_{\beta},e^{1}_{j} \rangle +\frac{2\lambda^{2}}{3} \langle e^{1}_{j},\eta^{1}_{\beta}\rangle^{2}+\frac{\lambda^{2}}{3}|e^{1}_{j}|^{2}|\eta^{1}_{\beta}|^{2}.
\end{equation}
Using Equation (\ref{hfubini}) in Equation (\ref{hbaseequation}),

\begin{equation}
\begin{aligned}
\label{hL0}
\sum_{A=1}^{m} -\langle N_{E_{A}},J_{\Sigma}(N_{E_{A}})\rangle &=\lambda^{2} \sum_{j=1}^{n}\sum_{\beta=1}^{d} \Big (\langle e^{1}_{j},\eta^{1}_{\beta}\rangle^{2}-\sum_{k=1}^{3}\langle e^{1}_{j},J_{k}(\eta^{1}_{\beta})\rangle^{2} \Big )\\
\displaystyle &=\lambda^{2} \Big (\sum_{j=1}^{n}\sum_{\beta=1}^{d}\langle e^{1}_{j},\eta^{1}_{\beta}\rangle^{2}-\sum_{j=1}^{n}\sum_{\beta=1}^{d}\sum_{k=1}^{3}\langle e^{1}_{j},J_{k}(\eta^{1}_{\beta})\rangle^{2} \Big )\\
\displaystyle &=\lambda^{2} \Big (\sum_{j=1}^{n}\sum_{\beta=1}^{d}\langle e^{1}_{j},\eta^{1}_{\beta}\rangle^{2} 
-\sum_{j=1}^{n}\sum_{\beta=1}^{d}\langle e^{1}_{j},J_{s}(\eta^{1}_{\beta})\rangle^{2}-\sum_{j=1}^{n}\sum_{\beta=1}^{d}\sum_{\substack{k=1\\
         k\neq s}}^{3}\langle e^{1}_{j},J_{k}(\eta^{1}_{\beta})\rangle^{2} \Big ).
\end{aligned}
\end{equation}

Ignoring the last term in the last equality and replacing $\beta$ with $k$ and $J_{s}$ with  $J$, we now proceed in the same way as in the proof of Lemma \ref{GENERALC} (from Equation (\ref{L0}) on) to study the first two terms in the last equality above. For $\beta\in\{1,\ldots,d\}$

\begin{equation*}
\begin{aligned}
\displaystyle |\eta^{1}_{\beta}|^{2}&=|J_{s}(\eta^{1}_{\beta})|^{2}=|(J_{s}(\eta^{1}_{\beta}),0)|^{2}\\
\displaystyle &=\sum_{j=1}^{n}\langle (J_{s}(\eta^{1}_{\beta}),0),e_{j}\rangle^{2}+\sum_{l=1}^{d}\langle (J_{s}(\eta^{1}_{\beta}),0),\eta_{l}\rangle^{2}\\
\displaystyle &=\sum_{j=1}^{n}\langle J_{s}(\eta^{1}_{\beta}),e_{j}^{1}\rangle^{2}+\sum_{l=1}^{d}\langle J_{s}(\eta^{1}_{\beta}),\eta^{1}_{l}\rangle^{2}.
\end{aligned}
\end{equation*}
Then,

\begin{center}
    $\displaystyle -\sum_{j=1}^{n}\langle J_{s}(\eta_{\beta}^{1}),e^{1}_{j} \rangle^{2}=-|\eta_{\beta}^{1}|^{2}+\sum_{l=1}^{d}\langle J_{s}(\eta_{\beta}^{1}) , \eta_{l}^{1} \rangle^{2}$,
\end{center}
and summing in $\beta$,

\begin{equation}
\label{hL1}-\sum_{\beta=1}^{d}\sum_{j=1}^{n}\langle J_{s}(\eta^{1}_{\beta}),e_{j}^{1}\rangle^{2}= -\sum_{\beta=1}^{d}|\eta^{1}_{\beta}|^{2}+\sum_{\beta=1}^{d}\sum_{l=1}^{d}\langle J_{s}(\eta^{1}_{\beta}),\eta^{1}_{l}\rangle^{2}.\\
\end{equation}
On the other hand, again for $\beta\in\{1,\ldots, d\}$

\begin{equation*}
\begin{aligned}
\displaystyle |\eta^{1}_{\beta}|^{2}&=|(\eta^{1}_{\beta},0)|^{2}\\
\displaystyle &=\sum_{j=1}^{n}\langle (\eta^{1}_{\beta},0),e_{j}\rangle^{2}+\sum_{l=1}^{d}\langle (\eta^{1}_{\beta},0),\eta_{l}\rangle^{2}\\
\displaystyle &=\sum_{j=1}^{n}\langle\eta^{1}_{\beta},e_{j}^{1}\rangle^{2}+\sum_{l=1}^{d}\langle\eta^{1}_{\beta},\eta^{1}_{l}\rangle^{2}.
\end{aligned}
\end{equation*}
Therefore,
\begin{center}
$\displaystyle \sum_{j=1}^{n}\langle \eta_{\beta}^{1},e_{j}^{1}\rangle^{2}=|\eta_{\beta}^{1}|^{2}-\sum_{l=1}^{d} \langle \eta_{\beta}^{1}, \eta_{l}^{1}\rangle^{2}$,
\end{center}
and summing in $\beta$,

\begin{equation}
\label{hL2}    
\displaystyle \sum_{\beta=1}^{d}\sum_{j=1}^{n}\langle\eta^{1}_{\beta},e_{j}^{1}\rangle^{2}=\sum_{\beta=1}^{d}|\eta^{1}_{\beta}|^{2}-\sum_{\beta=1}^{d}\sum_{l=1}^{d}\langle\eta^{1}_{\beta},\eta^{1}_{l}\rangle^{2}.
\end{equation}
Then, in order to prove Equation (\ref{hMAIN}), we replace Equations (\ref{hL1}) and (\ref{hL2}) in (\ref{hL0}), given us that

\begin{equation*}
\begin{aligned}
\displaystyle &\sum_{A=1}^{m}-\langle N_{E_{A}},J_{\Sigma}(N_{E_{A}})\rangle\\
\displaystyle&=\lambda^{2}\Big (\sum_{\beta=1}^{d}|\eta^{1}_{\beta}|^{2}-\sum_{\beta=1}^{d}\sum_{l=1}^{d}\langle\eta^{1}_{\beta},\eta^{1}_{l}\rangle^{2}
-\sum_{\beta=1}^{d}|\eta^{1}_{\beta}|^{2}+\sum_{\beta=1}^{d}\sum_{l=1}^{d}\langle J_{s}(\eta^{1}_{\beta}),\eta^{1}_{l}\rangle^{2}-\sum_{j=1}^{n}\sum_{\beta=1}^{d}\sum_{\substack{k=1\\
         k\neq s}}^{3}\langle e^{1}_{j},J_{k}(\eta^{1}_{\beta})\rangle^{2}\Big )
\\
\displaystyle&=\lambda^{2}\Big (-\sum_{\beta=1}^{d}\sum_{l=1}^{d}\langle\eta^{1}_{\beta},\eta^{1}_{l}\rangle^{2}+\sum_{\beta=1}^{d}\sum_{l=1}^{d}\langle J_{s}(\eta^{1}_{\beta}),\eta^{1}_{l}\rangle^{2}-\sum_{j=1}^{n}\sum_{\beta=1}^{d}\sum_{\substack{k=1\\
         k\neq s}}^{3}\langle e^{1}_{j},J_{k}(\eta^{1}_{\beta})\rangle^{2}\Big ).\\
\end{aligned}
\end{equation*}

Now, let us prove Equation (\ref{hMAIN2}). For $j\in\{1,\ldots,n\}$,

\begin{equation*}
\begin{aligned}
\displaystyle |e^{1}_{j}|^{2}&=|J_{s}(e^{1}_{j})|^{2}=|(J_{s}(e^{1}_{j}),0)|^{2}\\
\displaystyle &=\sum_{i=1}^{n}\langle (J_{s}(e^{1}_{j}),0),e_{i}\rangle^{2}+\sum_{\beta=1}^{d}\langle(J_{s}(e^{1}_{j}),0),\eta_{\beta}\rangle^{2}\\
\displaystyle &=\sum_{i=1}^{n}\langle J_{s}(e^{1}_{j}),e_{i}^{1}\rangle^{2}+\sum_{\beta=1}^{d}\langle J_{s}(e^{1}_{j}),\eta^{1}_{\beta}\rangle^{2}.
\end{aligned}
\end{equation*}
Then,
\begin{center}
    $\displaystyle -\sum_{\beta=1}^{d}\langle J_{s}(e^{1}_{j}), \eta_{\beta}^{1}\rangle^{2}=-|e_{j}^{1}|^{2}+\sum_{i=1}^{n}\langle J_{s}(e_{j}^{1}),e_{i}^{1}\rangle^{2}$,
\end{center}
and summing in $j$,

\begin{equation}
\label{hL3}
-\sum_{j=1}^{n}\sum_{\beta=1}^{d}\langle J_{s}(e^{1}_{j}),\eta^{1}_{\beta}\rangle^{2}
=-\sum_{j=1}^{n}|e^{1}_{j}|^{2}+\sum_{j=1}^{n}\sum_{i=1}^{n}\langle J_{s}(e^{1}_{j}),e_{i}^{1}\rangle^{2}.\\
\end{equation}
On the other hand, again for $j\in\{1,\ldots,n\}$

\begin{equation*}
\begin{aligned}
\displaystyle |e^{1}_{j}|^{2}&=|(e^{1}_{j},0)|^{2}\\
\displaystyle &=\sum_{i=1}^{n}\langle (e^{1}_{j},0),e_{i}\rangle^{2}+\sum_{\beta=1}^{d}\langle(e^{1}_{j},0),\eta_{\beta}\rangle^{2}\\
\displaystyle &=\sum_{i=1}^{n}\langle e^{1}_{j},e_{i}^{1}\rangle^{2}+\sum_{\beta=1}^{d}\langle e^{1}_{j},\eta^{1}_{\beta}\rangle^{2}.
\end{aligned}
\end{equation*}
Therefore,
\begin{center}
$\displaystyle\sum_{\beta=1}^{d}\langle e_{j}^{1}, \eta_{\beta}^{1} \rangle^{2}=|e_{j}^{1}|^{2}-\sum_{i=1}^{n}\langle e_{j}^{1},e_{i}^{1} \rangle^{2}$,    
\end{center}
summing in $j$,

\begin{equation}
\label{hL4} \sum_{j=1}^{n}\sum_{\beta=1}^{d}\langle e^{1}_{j},\eta^{1}_{\beta}\rangle^{2}
 =\sum_{j=1}^{n}|e^{1}_{j}|^{2}-\sum_{i=1}^{n}\sum_{j=1}^{n} \langle e^{1}_{i},e_{j}^{1}\rangle^{2}.
\end{equation}
Replacing Equations (\ref{hL3}) and (\ref{hL4}) in (\ref{hL0}),
\begin{equation*}
    \begin{aligned}
\displaystyle &\sum_{A=1}^{m} -\langle N_{E_{A}},J_{\Sigma}(N_{E_{A}})\rangle\\
\displaystyle & =\lambda^{2}(\sum_{j=1}^{n}|e^{1}_{j}|^{2}-\sum_{i=1}^{n}\sum_{j=1}^{n}\langle e^{1}_{i},e_{j}^{1}\rangle^{2}
-\sum_{j=1}^{n}|e^{1}_{j}|^{2}+\sum_{j=1}^{n}\sum_{i=1}^{n}\langle J_{s}(e^{1}_{j}),e_{i}^{1}\rangle^{2}-\sum_{j=1}^{n}\sum_{\beta=1}^{d}\sum_{\substack{k=1\\
         k\neq s}}^{3}\langle e^{1}_{j},J_{k}(\eta^{1}_{\beta})\rangle^{2})\\
\displaystyle\hspace{-0.1cm}&=\lambda^{2}(-\sum_{i=1}^{n}\sum_{j=1}^{n}\langle e^{1}_{i},e_{j}^{1}\rangle^{2}+\sum_{j=1}^{n}\sum_{i=1}^{n}\langle J_{s}(e^{1}_{j}),e_{i}^{1}\rangle^{2}-\sum_{j=1}^{n}\sum_{\beta=1}^{d}\sum_{\substack{k=1\\
         k\neq s}}^{3}\langle e^{1}_{j},J_{k}(\eta^{1}_{\beta})\rangle^{2}
).
\end{aligned}
\end{equation*}
\end{proof}
\end{subsection}

\subsection{CODIMENSION $1$ AND $2$}
\label{QC12}
The arguments that will be used in the proof of Theorem \ref{cod1h} are analogous to the arguments used in the proof of Theorem  \ref{cod1c}. The arguments that will be presented in the first part of the proof of Theorem \ref{hz1} are similar to the arguments presented in the proof of Lemma \ref{cod2c}. Then, after obtaining a characterization of the tangent space of the immersion, the second part of the proof follows similarly the proof of Theorem \ref{cod1c}.

\begin{theorem}
\label{cod1h}
Let $\Phi=(\psi,\phi):\Sigma \rightarrow \bar{M}:= \mathbb{H}P^{\frac{m_{1}}{4}}\times M$ be a compact stable minimal immersion of codimension $d=1$ and dimension $n$, where $M$ is any Riemannian manifold of dimension $m_{2}$. Then, $\Sigma=\mathbb{H}P^{\frac{m_{1}}{4}}\times \hat{\Sigma}$, $\Phi= Id \times \hat{\phi}$ where $\hat{\phi}:\hat{\Sigma}\rightarrow M$ is a compact stable minimal  immersion of codimension $1$, and therefore $\Phi(\Sigma)=\mathbb{H}P^{\frac{m_{1}}{4}}\times \hat{\phi}(\hat{\Sigma})$. In particular, for $m_{2}=1$, $\Sigma=\mathbb{H}P^{\frac{m_{1}}{4}}$, $\hat{\phi}$ is a constant function, and $\Phi(\Sigma)=\mathbb{H}P^{\frac{m_{1}}{4}}\times \{q\}$, for $q\in M$.
\end{theorem}
\begin{proof}
Since $d=1$, Equation (\ref{hMAIN}) becomes

\begin{equation*}
\begin{aligned}
\displaystyle\sum_{A=1}^{m} -\langle N_{E_{A}},J_{\Sigma}(N_{E_{A}})\rangle &=\lambda^{2}              \Big (
-\sum_{\substack{k=1\\
         k\neq s}}^{3}\sum_{j=1}^{n}\langle J_{k}(\eta^{1}),e^{1}_{j}\rangle^{2}\displaystyle +\langle J_{s}(\eta^{1}),\eta^{1}\rangle^{2}-\langle\eta^{1},\eta^{1}\rangle^{2}              \Big )
         \\
\displaystyle &=\lambda^{2} \Big (
-\sum_{\substack{k=1\\
         k\neq s}}^{3}\sum_{j=1}^{n}\langle J_{k}(\eta^{1}),e^{1}_{j}\rangle^{2}\displaystyle-|\eta^{1}|^{4}\Big ),
         \end{aligned}
\end{equation*}
where $\eta$ is an unitary vector at $N_{p}\Sigma$, for $p \in \Sigma$. The last expression differs from Equation (\ref{similar}) in the first term. This term does not add any new difficulties because it is non-positive. In fact, integrating both sides of the last equality, we have

\begin{center}
$\displaystyle 0\leq \sum_{A=1}^{m} -\int_{\Sigma}\langle N_{E_{A}},J_{\Sigma}(N_{E_{A}})\rangle d\Sigma=
-\lambda^{2}\int_{\Sigma}\sum_{\substack{k=1\\
         k\neq s}}^{3}\sum_{j=1}^{n}\langle J_{k}(\eta^{1}),e^{1}_{j}\rangle^{2}\displaystyle+|\eta^{1}|^{4}d\Sigma\leq 0,$
\end{center}
where we have used the fact that $\Sigma$ is stable in the first inequality. Hence, for $p\in \Sigma$, $\eta^{1}=0$, and therefore $\eta=(0,\eta^{2})$. Thus,

\begin{center}
    $d\Phi_{p}(T_{p}\Sigma)= \bar{D_{1}}(p) \bigoplus \bar{D_{2}}(p)$,
\end{center}
where $\bar{D_{1}}$ and $\bar{D_{2}}$ are given by

\begin{center}
    $\bar{D_{1}}(p)=\{(x,0): x\in T_{\psi(p)}\mathbb{H}P^{\frac{m_{1}}{2}}\} $\\
    
    $\bar{D_{2}}(p)= \{(0,w): w\in [\eta^{2}]^{\perp_{M}}\}$,
\end{center}
where $[z]^{\perp_{M}}$ is the orthogonal complement of $z$ in $T_{\phi(p)}M$. Since $\mathbb{H}P^{\frac{m_{1}}{4}}$ is simply connected, the rest of the proof follows by applying the same proof of Theorem \ref{cod1c}. 

\end{proof}

\begin{corollary}
There are no compact stable minimal hypersurfaces in the product manifold $\mathbb{H}P^{\frac{m_{1}}{4}}\times S^{s}$, $\mathbb{H}P^{\frac{m_{1}}{4}}\times \mathbb{O} P^{2}$, or $\mathbb{H}P^{\frac{m_{1}}{4}}\times \mathbb{K}P^{s}$, where $\mathbb{K}\in \{\mathbb{C}, \mathbb{H}\}$ other than  $\mathbb{H}P^{\frac{m_{1}}{4}}\times \{q\}$ in $\mathbb{H}P^{\frac{m_{1}}{4}}\times S^{1}$, where $q\in S^{1}$. 
\end{corollary}
\begin{proof}
Apply Theorem \ref{cod1h} and notice that there are no stable minimal hypersurfaces in $S^{s}$,  $\mathbb{C}P^{s}$, 
$\mathbb{H}P^{s}$, or $\mathbb{O} P^{2}$ (see Theorems \ref{t1.1} and \ref{ohnita}) other than a point $q$ in $S^{1}$.
\end{proof}

\begin{corollary}
The only compact stable minimal hypersurface in the product space $\mathbb{H}P^{\frac{m_{1}}{4}}\times \mathbb{R}P^{s}$ is $\mathbb{H}P^{\frac{m_{1}}{2}}\times \mathbb{R}P^{s-1}$.
\end{corollary}

\begin{proof}
Apply Theorem \ref{cod1h} and notice that the only stable minimal hypersurface in $\mathbb{R}P^{s}$ is
$\mathbb{R}P^{s-1}$ (see Theorem \ref{ohnita0}).
\end{proof}

\begin{theorem}
\label{hz1}
Let $\Phi=(\psi,\phi):\Sigma \rightarrow \mathbb{H}P^{\frac{m_{1}}{4}}\times M$ be a compact stable minimal immersion of codimension $d=2$ and dimension $n$, where $M$ is any Riemannian manifold of dimension $m_{2}$. Then, $\Sigma=\mathbb{H}P^{\frac{m_{1}}{4}}\times \hat{\Sigma}$, $\Phi= Id \times \hat{\phi}$ where $\hat{\phi}:\hat{\Sigma}\rightarrow M$ is a compact stable minimal immersion of codimension $2$, and therefore $\Phi(\Sigma)=\mathbb{H}P^{\frac{m_{1}}{4}}\times \hat{\phi}(\hat{\Sigma})$. In particular, for $m_{2}=1$, there are no compact stable minimal immersions of codimension $2$ in $\mathbb{H}P^{\frac{m_{1}}{4}}\times M$. And
for $m_{2}=2$, $\Sigma=\mathbb{H}P^{\frac{m_{1}}{4}}$, $\hat{\phi}$ is a constant function, and $\Phi(\Sigma)=\mathbb{H}P^{\frac{m_{1}}{4}}\times \{q\}$, for $q\in M$.
\end{theorem}

\begin{proof} Let $p\in \Sigma$. Since $d=2$, for $s\in \{1,2,3\}$, Equation (\ref{hMAIN}) becomes

\begin{equation}
\begin{aligned}
\label{hj1}
    &\sum_{A=1}^{m} -\langle N_{E_{A}},J_{\Sigma}(N_{E_{A}})\rangle \\
\displaystyle &=\lambda^{2}\Big (
-\sum_{\substack{k=1\\
         k\neq s}}^{3}\sum_{j=1}^{n}\sum_{\beta=1}^{2}\langle J_{k}(\eta^{1}_{\beta}),e^{1}_{j}\rangle^{2}\displaystyle +\sum_{\beta=1}^{2}\sum_{l=1}^{2}\langle J_{s}(\eta^{1}_{\beta}),\eta^{1}_{l}\rangle^{2}-\langle \eta^{1}_{\beta},\eta^{1}_{l}\rangle^{2}\Big )\\
         \hspace{-0.3cm}\displaystyle&=\lambda^{2}\Big (
-\sum_{\substack{k=1\\
         k\neq s}}^{3}\sum_{j=1}^{n}\sum_{\beta=1}^{2}\langle J_{k}(\eta^{1}_{\beta}),e^{1}_{j}\rangle^{2}\displaystyle- |\eta_{1}^{1}|^{4}-|\eta_{2}^{1}|^{4}-2\langle \eta_{1}^{1},\eta_{2}^{1}\rangle^{2}+2\langle J_{s}(\eta_{2}^{1}),\eta_{1}^{1}\rangle^{2} \Big ).
          \end{aligned}
\end{equation}
Notice that, ignoring the first term in the last equality, we have the same expression as in Equation (\ref{MAIND22}) (instead of $J$ we have the structure $J_{s}$). This term does not add any new difficulties because it is non-positive. From this point on, most of the arguments are analogous to those found in the proof of Lemma \ref{cod2c}.\\

If $\eta_{1}^{1}=0$, then 

\begin{equation}
\label{hj2}
    \displaystyle \sum_{A=1}^{m} -\langle N_{E_{A}},J_{\Sigma}(N_{E_{A}})\rangle=\lambda^{2}\Big (
-\sum_{\substack{k=1\\
         k\neq s}}^{3}\sum_{j=1}^{n}\langle J_{k}(\eta^{1}_{2}),e^{1}_{j}\rangle^{2}\displaystyle-|\eta_{2}^{1}|^{4}\Big )\leq 0.
\end{equation}

If $\eta_{1}^{1}\neq 0$, we can write $\eta_{2}^{1}$ in terms of $\eta_{1}^{1}$, $J_{s}(\eta_{1}^{1})$, and $X$ for some $X\in T_{\psi(p)}\mathbb{H}P^{\frac{m_{1}}{4}}$ which is unitary and orthogonal to $\eta_{1}^{1}$ and $J_{s}(\eta_{1}^{1})$. Then, 
\begin{center}
    $\displaystyle \langle \eta_{2}^{1},J_{s}(\eta_{1}^{1})\rangle^{2}= |\eta_{1}^{1}|^{2}|\eta_{2}^{1}|^{2}-\langle \eta_{2}^{1},\eta_{1}^{1}\rangle^{2}-|\eta_{1}^{1}|^{2}\langle\eta_{2}^{1},X\rangle^{2}.$
\end{center}
Replacing the last equation in Equation (\ref{hj1}), we have 

\begin{equation}
\label{hj3}
    \sum_{A=1}^{m} -\langle N_{E_{A}},J_{\Sigma}(N_{E_{A}})\rangle
\end{equation}

\begin{center}   $\displaystyle=\lambda^{2}\Big (
-\sum_{\substack{k=1\\
         k\neq s}}^{3}\sum_{j=1}^{n}\sum_{\beta=1}^{2}\langle J_{k}(\eta^{1}_{\beta}),e^{1}_{j}\rangle^{2}-(|\eta_{1}^{1}|^{2}-|\eta_{2}^{1}|^{2})^{2}-4\langle\eta_{1}^{1},\eta_{2}^{1}\rangle^{2}
-2|\eta_{1}^{1}|^{2}\langle \eta_{2}^{1},X\rangle^{2}\Big )\leq 0.
          $
\end{center}

From Equations (\ref{hj2}) and (\ref{hj3}), we have 
\begin{center}
    $\displaystyle \sum_{A=1}^{m} -\langle N_{E_{A}},J_{\Sigma}(N_{E_{A}})\rangle\leq 0$.
\end{center}
We follow the proof of Lemma \ref{cod2c} (from Equation (\ref{signo}) on) to obtain,
\begin{equation}
\label{hj4}
    \eta_{2}^{1}=\pm J_{s}(\eta_{1}^{1}),
\end{equation}
for $s\in\{1,2,3\}$. \\

From Equation (\ref{hj4}), we have

\begin{center}
    $\langle J_{1}(\eta_{1}^{1}),J_{2}(\eta_{1}^{1})\rangle=\langle\pm\eta_{2}^{1},\pm\eta_{2}^{1}\rangle=\pm |\eta_{2}^{1}|^{2}$.
\end{center}
On the other hand,

\begin{center}
    $\langle J_{1}(\eta_{1}^{1}),J_{2}(\eta_{1}^{1})\rangle=-\langle J_{2}(J_{1}(\eta_{1}^{1})),\eta_{1}^{1}\rangle=\langle J_{3}(\eta_{1}^{1}),\eta_{1}^{1}\rangle=0$.
\end{center}
Hence $\eta_{2}^{1}=0$. Since $J_{s}$ is an isometry from Equation (\ref{hj4}), $\eta_{1}^{1}=0$. Therefore, at $p\in \Sigma$, we have that $\eta_{1}=(0, \eta_{1}^{2})$ and $\eta_{2}=(0, \eta_{2}^{2})$. Thus,

\begin{center}
    $ d\Phi_{p}(T_{p}\Sigma)= \bar{D_{1}}(p) \bigoplus \bar{D_{2}}(p)$,
\end{center}
where $\bar{D_{1}}$ and $\bar{D_{2}}$ are given by: 

\begin{center}
    $\bar{D_{1}}(p)=\{(x,0): x\in T_{\psi(p)}\mathbb{H}P^{\frac{m_{1}}{4}}\} $\\
    
    $\bar{D_{2}}(p)= \{(0,w): w\in [Gen\{\eta_{1}^{2}, \eta_{2}^{2}\}]^{\perp_{M}}  \}$,
\end{center}
where $[.]^{\perp_{M}}$ is the orthogonal complement in $T_{\phi(p)}M$ and $Gen\{\eta_{1}^{2}, \eta_{2}^{2}\}$ is the subspace generated by the vectors $\eta_{1}^{2}, \eta_{2}^{2}$ in  $T_{\phi(p)}M$. Since $\mathbb{H}P^{\frac{m_{1}}{4}}$ is simply connected, the rest of the proof follows by applying the same proof of Theorem \ref{cod1c}. 
\end{proof}

\begin{corollary}
There are no compact stable minimal immersions of codimension $d=2$ in the product manifold $\mathbb{H}P^{\frac{m_{1}}{4}}\times S^{s}$, $\mathbb{H}P^{\frac{m_{1}}{4}}\times \mathbb{O} P^{2}$, or $\mathbb{H}P^{\frac{m_{1}}{4}}\times \mathbb{H}P^{s}$ other than  $\mathbb{H}P^{\frac{m_{1}}{4}}\times \{q\}$ in  $\mathbb{H}P^{\frac{m_{1}}{4}}\times S^{2}$, where $q\in S^{2}$. 
\end{corollary}
\begin{proof}
Apply Theorem \ref{hz1} and notice that there are no stable minimal immersions of codimension $2$ in $S^{s}$, $\mathbb{H}P^{s}$, or $\mathbb{O} P^{2}$ (see Theorems \ref{t1.1} and \ref{ohnita}) other than a point $q$ in $S^{2}$.
\end{proof}

\begin{corollary}
The only compact stable minimal immersion of codimension $2$ in the product space $\mathbb{H}P^{\frac{m_{1}}{4}}\times \mathbb{R}P^{s}$ is $\mathbb{H}P^{\frac{m_{1}}{2}}\times \mathbb{R}P^{s-2}$, and in the product space $\mathbb{H}P^{\frac{m_{1}}{4}}\times \mathbb{C}P^{s}$ is  $\mathbb{H}P^{\frac{m_{1}}{4}}\times M^{2s-2}$, where $M$ is a complex submanifold of dimension $2s-2$ immersed in  $\mathbb{C}P^{s}$.
\end{corollary}

\begin{proof}
Apply Theorem \ref{hz1} and notice that the only stable minimal immersions of codimension $2$ in $\mathbb{R}P^{s}$ and $\mathbb{C}P^{s}$  are
$\mathbb{R}P^{s-2}$ and $M$, respectively, where $M$ is a complex submanifold of dimension $2s-2$ immersed in  $\mathbb{C}P^{s}$ (see Theorems \ref{ohnita0} and \ref{ohnita}).
\end{proof}

\subsection{DIMENSION $1$ AND $2$}
\label{QD12}
The arguments that will be used in the proof of Theorem \ref{dim1h} are analogous to the arguments used in the proof of Theorem  \ref{dim1c}. The arguments that will be presented in the first part of the proof of Theorem \ref{dim2h} are similar to the arguments presented in the proof of Lemma \ref{cod2c}. Then, after obtaining a characterization of the tangent space of the immersion, the second part of the proof follows similarly the proof of Theorem \ref{dim1c}.

\begin{theorem}
\label{dim1h}
Let $\Phi=(\psi,\phi):\Sigma \rightarrow \mathbb{H}P^{\frac{m_{1}}{4}}\times M$ be a compact stable minimal immersion of dimension $n=1$ and codimension $d$, where $M$ is any Riemannian manifold of dimension $m_{2}$. Then, $\phi:\Sigma \rightarrow M$ is a  stable geodesic, $\psi$ is a constant function, and therefore $\Phi(\Sigma)=\{r\}\times \phi(\Sigma)$ with $r$ a point of $\mathbb{H}P^{\frac{m_{1}}{4}}$.

\end{theorem}

\begin{proof}
Since $n=1$, Equation (\ref{hMAIN2}) becomes

\begin{equation}
\label{hG0}
\sum_{A=1}^{m} -\langle N_{E_{A}},J_{\Sigma}(N_{E_{A}})\rangle=\lambda^{2}\Big (-\sum_{\substack{k=1\\
         k\neq s}}^{3}\sum_{\beta=1}^{d}\langle J_{k}(\eta^{1}_{\beta}),e^{1}\rangle^{2}-|e^{1}|^{4}\Big ),
\end{equation}
where $e=(e^{1},e^{2})$ is a unit vector in $d\Phi_{p}(T_{p}\Sigma)$, for $p\in\Sigma$. The last expression just differs from Equation (\ref{G0}) in the first term. As in the last subsection, this term does not add any new difficulties because it is non-positive. We have:
 
\begin{center}
$\displaystyle 0\leq  \sum_{A=1}^{m} -\int_{\Sigma}\langle N_{E_{A}},J_{\Sigma}(N_{E_{A}})\rangle d\Sigma=-\lambda^{2}\int_{\Sigma} \sum_{\substack{k=1\\
         k\neq s}}^{3}\sum_{\beta=1}^{d}\langle J_{k}(\eta^{1}_{\beta}),e^{1}\rangle^{2}+|e^{1}|^{4}d\Sigma\leq 0,$
\end{center}
where we have used the fact that $\Sigma$ is stable in the first inequality. Hence, for $p\in\Sigma$, $e^{1}=0$. Therefore, 

\begin{center}
    $ d\Phi_{p}(T_{p}\Sigma)= \{ \alpha (0,e^{2}):\alpha\in \mathbb{R}\}$,
\end{center}
and thus,
\begin{center}
    $N_{p}\Sigma = \{ (v,0): v\in T_{\psi(p)}\mathbb{H}P^{\frac{m_{1}}{4}}\}\bigoplus \{(0,w): w\in [e^{2}]^{\perp_{M}}\}$.
\end{center}
The rest of the proof follows in a similar way as in the proof of Theorem \ref{dim1c}.
\end{proof}

\begin{corollary}
There are no compact stable geodesics in the product space $\mathbb{H}P^{\frac{m_{1}}{4}}\times S^{s}$, $\mathbb{H}P^{\frac{m_{1}}{4}}\times \mathbb{O} P^{2}$, or  $\mathbb{H}P^{\frac{m_{1}}{4}}\times \mathbb{K}P^{s}$, where $\mathbb{K}\in \{\mathbb{C}, \mathbb{H}\}$ other than $\{r\}\times S^{1}$ in $\mathbb{H}P^{\frac{m_{1}}{4}}\times S^{1}$, where $r\in \mathbb{H}P^{\frac{m_{1}}{4}}$.
\end{corollary}
\begin{proof}
Apply Theorem \ref{dim1h} and notice that there are no stable geodesics in $S^{s}$, $\mathbb{C}P^{s}$, 
$\mathbb{H}P^{s}$, or $\mathbb{O} P^{2}$ (see Theorems \ref{t1.1} and \ref{ohnita}) other than $S^{1}$ in $S^{1}$.
\end{proof}

\begin{corollary}
The only compact stable geodesic in the product space $\mathbb{H}P^{\frac{m_{1}}{4}}\times \mathbb{R}P^{s}$ is $\{r\}\times \mathbb{R}P^{1}$, $r\in \mathbb{H}P^{\frac{m_{1}}{4}}$.
\end{corollary}

\begin{proof}
Apply Theorem \ref{dim1h} and notice that the only stable geodesic in $\mathbb{R}P^{s}$ is
$\mathbb{R}P^{1}$ (see Theorem \ref{ohnita0}).
\end{proof}


\begin{theorem}
\label{dim2h}
Let $\Phi=(\psi,\phi):\Sigma \rightarrow \mathbb{H}P^{\frac{m_{1}}{4}}\times M$ be a compact stable minimal immersion of dimension $n=2$ and codimension $d$, where $M$ is any Riemannian manifold of dimension $m_{2}$. Then, $\phi:\Sigma \rightarrow M$ is a  stable minimal immersion of dimension $2$, $\psi$ is a constant function, and therefore $\Phi(\Sigma)=\{r\}\times \phi(\Sigma)$ with $r$ a point of $\mathbb{H}P^{\frac{m_{1}}{4}}$.
\end{theorem}

\begin{proof}
Let $p\in \Sigma$. Since n=2, for $s\in \{1,2,3\}$, Equation (\ref{hMAIN2}) becomes

\begin{equation}
\begin{aligned}
\label{hk1}
\displaystyle \sum_{A=1}^{m} -\langle N_{E_{A}},J_{\Sigma}(N_{E_{A}})\rangle &=\lambda^{2}\Big (
-\sum_{\substack{k=1\\
         k\neq s}}^{3}\sum_{j=1}^{2}\sum_{\beta=1}^{d}\langle J_{k}(\eta^{1}_{\beta}),e^{1}_{j}\rangle^{2}\displaystyle +\sum_{i=1}^{2}\sum_{j=1}^{2}\langle J_{s}(e^{1}_{i}),e^{1}_{j}\rangle^{2}-\langle e^{1}_{i},e^{1}_{j}\rangle^{2}\Big )\\
        \displaystyle &=\lambda^{2} \Big (
-\sum_{\substack{k=1\\
         k\neq s}}^{3}\sum_{j=1}^{2}\sum_{\beta=1}^{d}\langle J_{k}(\eta^{1}_{\beta}),e^{1}_{j}\rangle^{2}\displaystyle- |e_{1}^{1}|^{4}-|e_{2}^{1}|^{4}-2\langle e_{1}^{1},e_{2}^{1}\rangle^{2}+2\langle J_{s}(e_{2}^{1}),e_{1}^{1}\rangle^{2}\Big ).
         \end{aligned}
\end{equation}

If $e_{1}^{1}=0$,

\begin{equation}
\label{hk2}
    \displaystyle \sum_{A=1}^{m} -\langle N_{E_{A}},J_{\Sigma}(N_{E_{A}})\rangle=\lambda^{2}\Big (
-\sum_{\substack{k=1\\
         k\neq s}}^{3}\sum_{\beta=1}^{d}\langle J_{k}(\eta^{1}_{\beta}),e^{1}_{2}\rangle^{2}\displaystyle-|e_{2}^{1}|^{4}\Big )\leq 0.
\end{equation}

If $e_{1}^{1}\neq 0$, we can write $e_{2}^{1}$ in terms of $e_{1}^{1}$, $J_{s}(e_{1}^{1})$, and $X$ for some vector $X\in T_{\psi(p)}\mathbb{H}P^{\frac{m_{1}}{4}}$ which is unitary and orthogonal to $e_{1}^{1}$ and $J_{s}(e_{1}^{1})$. Then,

\begin{center}
    $\displaystyle \langle e_{2}^{1},J_{s}(e_{1}^{1})\rangle^{2}= |e_{1}^{1}|^{2}|e_{2}^{1}|^{2}-\langle e_{2}^{1},e_{1}^{1}\rangle^{2}-|e_{1}^{1}|^{2}\langle e_{2}^{1},X \rangle^{2}.$
\end{center}
Replacing the last equation in Equation (\ref{hk1}), we have 

\begin{equation}
\begin{aligned}
\label{hk3}
    &\sum_{A=1}^{m} -\langle N_{E_{A}},J_{\Sigma}(N_{E_{A}})\rangle\\
\displaystyle &=\lambda^{2}\Big (
-\sum_{\substack{k=1\\
         k\neq s}}^{3}\sum_{j=1}^{2}\sum_{\beta=1}^{d}\langle J_{k}(\eta^{1}_{\beta}),e^{1}_{j}\rangle^{2}-(|e_{1}^{1}|^{2}-|e_{2}^{1}|^{2})^{2}  -4\langle e_{1}^{1},e_{2}^{1}\rangle^{2}
-2|e_{1}^{1}|^{2}\langle e_{2}^{1},X\rangle^{2}\Big )\leq 0.
          \end{aligned}
\end{equation}

From Equations (\ref{hk2}) and (\ref{hk3}), we have 
\begin{center}
    $\displaystyle \sum_{A=1}^{m} -\langle N_{E_{A}},J_{\Sigma}(N_{E_{A}})\rangle\leq 0$.
\end{center}
We follow the proof of Lemma \ref{cod2c} (from Equation (\ref{signo}) on) to obtain

\begin{equation}
\label{hk4}
    e_{2}^{1}=\pm J_{s}(e_{1}^{1})
\end{equation}
for $s\in\{1,2,3\}$.\\

From Equation (\ref{hk4}), we have

\begin{center}
    $\langle J_{1}(e_{1}^{1}),J_{2}(e_{1}^{1})\rangle=\langle \pm e_{2}^{1},\pm e_{2}^{1}\rangle=\pm |e_{2}^{1}|^{2}$.
\end{center}
On the other hand

\begin{center}
    $\langle J_{1}(e_{1}^{1}),J_{2}(e_{1}^{1})\rangle=-\langle J_{2}(J_{1}(e_{1}^{1})),e_{1}^{1}\rangle=\langle J_{3}(e_{1}^{1}),e_{1}^{1}\rangle=0$.
\end{center}
Hence $e_{2}^{1}=0$. Since $J_{s}$ is an isometry from Equation (\ref{hk4}), $e_{1}^{1}=0$. Therefore, at $p\in \Sigma$ we have that $e_{1}=(0, e_{1}^{2})$ and $e_{2}=(0, e_{2}^{2})$. Thus

\begin{center}
    $ d\Phi_{p}(T_{p}\Sigma)= \{ (0,x): x \in Gen \{e^{2}_{1},e^{2}_{2}\}\}$,
\end{center}
and then,
\begin{center}
    $N_{p}\Sigma = \{ (v,0): v\in T_{\psi(p)}\mathbb{H}P^{\frac{m_{1}}{4}}\}\bigoplus \{(0,w): w\in [Gen \{e^{2}_{1},e^{2}_{2}\}]^{\perp_{M}}\}$.
\end{center}

where $[.]^{\perp_{M}}$ is the orthogonal complement in $T_{\phi(p)}M$ and $Gen\{e_{1}^{2}, e_{2}^{2}\}$ is the subspace generated by the vectors $e_{1}^{2}, e_{2}^{2}$ in  $T_{\phi(p)}M$. The rest of the proof follows by applying the same proof of Theorem \ref{dim1c}.
\end{proof}

\begin{corollary}
There are no compact stable minimal immersions of dimension $2$ in the product manifold $\mathbb{H}P^{\frac{m_{1}}{4}}\times S^{s}$, $\mathbb{H}P^{\frac{m_{1}}{4}}\times \mathbb{O} P^{2}$, or $\mathbb{H}P^{\frac{m_{1}}{4}}\times \mathbb{H}P^{s}$ other than $\{r\}\times S^{2}$ in $\mathbb{H}P^{\frac{m_{1}}{4}}\times S^{2}$, where $r\in \mathbb{H}P^{\frac{m_{1}}{4}}$.
\end{corollary}
\begin{proof}
Apply Theorem \ref{dim2h} and notice that there are no stable minimal immersions of dimension $2$ in $S^{s}$, $\mathbb{H}P^{s}$, or $\mathbb{O} P^{2}$ (see Theorems \ref{t1.1} and \ref{ohnita}) other than $S^{2}$ in $S^{2}$.
\end{proof}

\begin{corollary}
The only compact stable minimal immersion of dimension $2$ in the product space $\mathbb{H}P^{\frac{m_{1}}{4}}\times \mathbb{R}P^{s}$ is $ \{r\}\times \mathbb{R}P^{2}$, and in $\mathbb{H}P^{\frac{m_{1}}{4}}\times \mathbb{C}P^{s}$ is $\{r\}\times M$,  where $M$ is a complex submanifold of dimension $2$ immersed in  $\mathbb{C}P^{s}$ and $r\in \mathbb{H}P^{\frac{m_{1}}{4}}$
\end{corollary}

\begin{proof}
Apply Theorem \ref{dim2h} and notice that the only stable minimal immersions of dimension $2$ in $\mathbb{R}P^{s}$ and $\mathbb{C}P^{s}$  are
$\mathbb{R}P^{2}$ and $M$ respectively,  where $M$ is a complex submanifold of dimension $2$ immersed in  $\mathbb{C}P^{s}$  (see Theorems \ref{ohnita0} and \ref{ohnita}).
\end{proof}

\bibliographystyle{plain}
\bibliography{references}

\end{document}